\documentclass[preprint,12pt]{elsarticle}
\usepackage[singlelinecheck=false,labelsep=period,font=footnotesize]{caption}
\usepackage{comment}
\usepackage{mathtools}
\usepackage{empheq}  
\usepackage{float}
\usepackage{multicol}
\usepackage{multirow}
\usepackage{xr-hyper}
\usepackage{hyperref}
\usepackage{graphicx}
\usepackage{amsfonts}
\usepackage{amsmath, amsthm}
\usepackage{color}
\usepackage[normalem]{ulem} 
\usepackage{epstopdf}
\usepackage{subcaption}
\usepackage{eurosym}
\usepackage{enumitem}
\usepackage{bbm}
\usepackage{pgfplots,pgfplotstable,booktabs}
\pgfplotsset{compat=1.8}
\usepgfplotslibrary{statistics}
\usepackage{balance}

\newtheorem{thm}{Theorem}

\newtheorem{prop}[thm]{Proposition}

\begin{document}

\begin{frontmatter}

\title{A bilevel framework for decision-making under uncertainty with contextual information}

\author{M. A. Mu\~noz, S. Pineda,  J. M. Morales}



\address{OASYS Group, University of Malaga, Malaga, Spain}


\begin{abstract}
In this paper, we propose a novel approach for data-driven decision-making under uncertainty in the presence of contextual information. Given a finite collection of observations of the uncertain parameters
and potential explanatory variables (i.e., the contextual information), our approach fits a parametric model to those data that is specifically tailored to maximizing the decision value, while accounting for possible feasibility constraints. From a mathematical point of view, our framework translates into a bilevel program, for which we provide both a fast regularization procedure and a big-M-based reformulation that can be solved using off-the-shelf optimization solvers. We showcase the benefits of moving from the traditional scheme for model estimation (based on statistical quality metrics) to decision-guided prediction using three different practical problems. We also compare our approach with existing ones in a realistic case study that considers a strategic power producer that participates in the Iberian electricity market. Finally, we use these numerical simulations to analyze the conditions (in terms of the firm's cost structure and production capacity) under which our approach proves to be more advantageous to the producer. 
\end{abstract}


\end{frontmatter}


\section{Introduction}
\label{sec:introduction}
In the last couple of decades, the field of decision-making under uncertainty has regained momentum, spurred by the new opportunities that the Digital Age has brought to modern economies. As a result, this field has been prolific in the design and development of new tools capable of exploiting the vast amount of information that human societies currently generate, compile and record, mainly in the form of \emph{data}.

From among all the exciting advances that 
have been achieved in the realm of decision making under uncertainty in recent years, we highlight the so-called \emph{data-driven optimization under uncertainty}, which endows the decision maker with a powerful and versatile mathematical framework to hedge her decisions against both the intrinsic risk of an uncertain world  and the limited and incomplete knowledge of the random phenomena that can be retrieved from a finite set of observations or data.

Data-driven optimization under uncertainty has been applied to a broad range of contexts and problems, for instance, inventory management~\citep{Bertsimas2020, Donti2017}, nurse staffing~\citep{Ban2019}, portfolio optimization~\citep{Bertsimas2019dynamic, Elmachtoub2017, esfahani2018data}, shipment planning~\citep{Bertsimas2019dynamic}, network flow~\citep{Elmachtoub2017}, power dispatch~\citep{li2019ambiguous}, and vehicle routing~\citep{Hao2020}, just to name a few. For a recent survey on the topic and its applications, we refer the reader to \citep{keith2019survey} and \citep{Bakker2019structuring}. 

In this paper, we first compare the proposed methodology with existing ones using two classical conditional stochastic optimization problems, namely, the newsvendor problem~\cite{Ban2019, Mundru2019, bertsimas2018} and the product placement problem~\citep{Bertsimas2020}. Additionally, we consider the problem of a strategic firm that has to decide the generation quantity that maximizes its expected profit while facing the uncertainty related to market conditions. This problem has a long tradition in the Economics and Management Science literature (see, for instance, \citep{Vives1984, wu2008incentives, Bimpikis2019}). In particular, we take \emph{electricity} as the homogeneous good to be produced and thus, we place ourselves in the context of electricity markets, where this problem has received a great deal of attention since the deregularization of the power sector \citep{Allaz1993, Ruiz2008}. Most existing models address this problem by forecasting, as accurately as possible, the electricity market behavior. Then, such forecasts are used to compute the decision that maximizes the producer's profit. Here we present a novel and alternative data-driven procedure that considers the problem structure and leverage available auxiliary data to enhance market participation and increase profits. Our approach is formulated as a bilevel program that, under convexity assumptions, can be efficiently solved using commercially available optimization solvers. We demonstrate the superior performance of the proposed approach on a realistic case study that uses data from the Iberian electricity market.

In short, our contributions are threefold, namely:
\begin{itemize}
    \item[-] From a methodological point of view, we propose a novel data-driven framework for conditional stochastic optimization, whereby the parameters that are input to the decision-making problem are formulated as a function of some \emph{covariates} or \emph{features}. This function is, in turn, estimated factoring in its impact on the decision value. \textcolor{black}{Finally, by way of this function, we construct a deterministic (single-scenario) surrogate optimization model that delivers decisions that are good in terms of the original conditional stochastic program.} In Section~\ref{sec:math_framework}, we introduce and mathematically formalize our proposal along with alternative state-of-the-art approaches available in the technical literature. Our approach is formulated as a bilevel optimization problem that can be reformulated as a single-level optimization problem and solved using off-the-shelf optimization solvers as discussed in Section \ref{subsec:solution_strategy}.
    \item[-] From a theoretical perspective, we compare our approach with existing ones in Section \ref{sec:application} for three different applications, namely, the newsvendor problem, the product placement problem, and the strategic producer problem.
    \item[-] From a more practical point of view, Section~\ref{sec:simulations} provides simulation results for the strategic producer problem using both an illustrative example and a realistic case study based on the Iberian electricity market. The numerical experiments show that our proposal can significantly increase the competitive edge of the strategic producer depending on her cost structure and the market demand elasticity.
\end{itemize}

We conclude the paper with a brief compilation of the most relevant observations in Section~\ref{sec:conclusions}.

\section{Mathematical framework and related work}
\label{sec:math_framework}
In decision making we often model the uncertainty as a random vector of parameters ($y\in\mathcal{Y}\subseteq\mathbb{R}^m$) governed by a real unknown distribution $Y$ and, typically, some relevant contextual information $(x \in \mathcal{X} \subseteq \mathbb{R}^{p}) \sim X$ is available before the decision is to be made. Following this scenario, the decision maker is interested in solving the \emph{conditional} stochastic optimization problem
\begin{align}
    \min_{z \in Z} & \enskip \mathbb{E} [f(z; Y) | X = x] \label{eq:ms0_obj} 
\end{align}
\noindent where $f: \mathbb{R}^n \times \mathbb{R}^m \rightarrow \mathbb{R}$ is a known function in the decision $z \in \mathbb{R}^n$, and $Z \subseteq \mathbb{R}^n$ is a nonempty, compact set known with certainty (i.e., independent of $Y$), to which the decision $z$ must belong. In practice, neither the \emph{joint} distribution of $X$ and $Y$, nor the conditional distribution of $Y$ given $X = x$ are known and therefore, problem~\eqref{eq:ms0_obj} cannot be solved. \label{paragraph:rev1_R3.2} On top of that, even if the true distribution were known and the decision $z$ were fixed, problem~\eqref{eq:ms0_obj} would typically require to compute the expectation of a function of a continuous random vector (i.e., a multivariate integral), which is, in itself, a hard task in general. Instead, the only information that the decision maker typically has is a sample $S = \left\lbrace (y_{i}, x_{i}), \forall i \in \mathcal{N} \right\rbrace $ where $y_{i} \in \mathbb{R}^m$ is a particular outcome of the uncertainty $Y$ recorded under the context $x_{i} \in \mathcal{X}$, and $\mathcal{N}$ denotes the set of available samples. 

Against this background, problem~\eqref{eq:ms0_obj} is alternatively replaced with a \emph{surrogate} optimization problem, in the hope that the solution to the latter is good enough for the former. In this line, different approaches have been proposed to construct such a surrogate optimization problem. For instance, the traditional \emph{modus operandi} follows the rule ``first predict, then optimize,'' which results in the following \textit{surrogate problem} to approximate the solution to \eqref{eq:ms0_obj}:
\begin{align}
    \min_{z \in Z} & \enskip f(z; \hat{y})  \label{eq:surrogate} 
\end{align}
where $\hat{y}$ denotes an estimate of the outcome of the uncertainty $Y$ under the contextual information $x \in \mathcal{X} \subseteq \mathbb{R}^{p}$. The surrogate problem~\eqref{eq:surrogate} is attractive for several reasons. First and foremost, it is much simpler and faster to solve than~\eqref{eq:ms0_obj}. Actually, it is a \emph{deterministic} optimization problem that, as opposed to~\eqref{eq:ms0_obj},  only requires evaluating the cost function $f(z; \cdot)$ at the single value or scenario $\hat{y}$. Furthermore, problem~\eqref{eq:surrogate} seems intuitive and natural, especially when $\hat{y}$  represents ``the most likely value'' for $Y$ given $X = x$. Indeed, the  single scenario $\hat{y}$ is often chosen as an estimate of the expected value of the uncertainty $Y$ conditional on $X = x$, that is, $\hat{y} \approx \mathbb{E}[Y|X=x]$, where, logically, the approximation is built from the available sample $S = \left\lbrace (y_{i}, x_{i}), \forall i \in \mathcal{N} \right\rbrace$. In the realm of forecasting, the estimate $\hat{y}$  is usually referred to as a \emph{point prediction}. 

In order to build the estimate $\hat{y} \approx \mathbb{E}[Y|X=x]$, a function $g^{\text {FO}}: \mathcal{X} \times \mathbb{R}^q \rightarrow \mathbb{R}^m$ is normally chosen from a $w$-parameterized family $G^{\text {FO}}$, with $w \in \mathbb{R}^q$, to construct the forecasting model $\hat{y} = g^{\text {FO}}(x; w)$. The goodness of a certain parameter vector $w$ is quantified in terms of a loss function $l^{\text {FO}}(y, \hat{y}): \mathcal{Y} \times \mathbb{R}^m \rightarrow \mathbb{R}$ that measures the accuracy of the estimate. 
Then, given the sample $S = \left\lbrace (y_{i}, x_{i}), \forall i \in \mathcal{N} \right\rbrace$, the choice of $w$ is driven by the minimization of the \emph{in-sample} loss, as expressed below: 
\begin{align}
    w^{\text {FO}} \in \arg \min_{w \in \mathbb{R}^q} \sum_{i \in \mathcal{N}} l^{\text {FO}}(y_{i},g^{\text {FO}}(x_{i}; w)) \label{eq:forecasting_step_theo}
\end{align}

In this framework, the optimal decision $z^{\text {FO}}$ under the context $X = x$ is thus obtained by solving the following deterministic problem:
\begin{align}
    z^{\text {FO}}(x) \in \arg \min_{z \in Z} & \enskip f(z; g^{\text {FO}}(x;w^{\text {FO}})) \label{eq:md1_obj} 
\end{align}

We refer to this approach, which relies on a \emph{good} forecast of the uncertainty $Y$ (in particular, an estimate of $\mathbb{E}[Y|X=x]$), as FO (short for FOrecasting). Even though this approach is intuitive and may perform relatively well in many situations, it is fundamentally flawed for the following two basic reasons. First, since $\hat{y} \approx \mathbb{E}[Y|X=x]$ in FO, the surrogate problem~\eqref{eq:surrogate} works as a proxy of the problem  
\begin{align}
    \min_{z \in Z} & \enskip f(z; \mathbb{E}[Y|X=x]) 
\end{align}
which, in general, is not equivalent to~\eqref{eq:ms0_obj}. Second, even in those cases where these two problems are indeed equivalent, the loss function $l^{\text {FO}}$ that is typically used to compute $w^{\text {FO}}$ (for example, the squared error) is solely intended to get a statistically good estimate of $\mathbb{E}[Y|X=x]$ and does not account for the nominal objective $f$ or the constraints that the decision $z$ must satisfy. For instance, approach \eqref{eq:forecasting_step_theo}-\eqref{eq:md1_obj} is unable to capture that overestimating $\mathbb{E}[Y|X=x]$ might worsen the objective function $f$ much more than underestimating it.

In view of these design flaws, a number of works have proposed to replace the problem-agnostic $l^{\text {FO}}$ that is generally used in \eqref{eq:forecasting_step_theo} with a problem-aware loss function $l^{\text {SP}}(y,\hat{y}) = f(\dot{z}(\hat{y}); y)$ where $l^{\text {SP}}: \mathbb{R}^m \times \mathbb{R}^m \rightarrow \mathbb{R}$ and $\dot{z}: \mathcal{Y} \rightarrow Z$ defined as $\dot{z}(y) = \arg \min_{z \in Z} f(z; y)$. Therefore, function $l^{\text {SP}}$ evaluates the loss of optimality associated with the decision $\dot{z}(\hat{y})$ that is prescribed by the surrogate decision-making problem \eqref{eq:surrogate} for the single value $\hat{y}$. Accordingly, the optimal parameter vector $w^{\text {SP}}$ is obtained as the one that minimizes the \emph{in-sample} optimality loss, that is:
\begin{align}
    w^{\text {SP}} \in \arg \min_{w \in \mathbb{R}^q} & \sum_{i \in \mathcal{N}} f(\dot{z}(g^{\text {SP}}(x_{i};w)); y_{i})  \label{eq:w_spto}
\end{align}
\noindent where the function $g^{\text {SP}}: \mathcal{X} \times \mathbb{R}^q \rightarrow \mathbb{R}^m$ is chosen from a family of functions $G^{\text {SP}}$. We use the acronym SP, which stands for ``Smart Predict'', to refer to this setup. Solving problem \eqref{eq:w_spto} using descent optimization methods requires to compute the gradient of the loss function $l^{\text {SP}}(y,\hat{y})$ with respect to $w$. This may not be feasible, since it involves the differentiation of the discontinuous function $\dot{z}(y)$ \cite{Mandi2019}. To overcome this difficulty, a great deal of research has been devoted to finding methods to approximate the gradient of \eqref{eq:w_spto} for particular instances. The work developed in \cite{Kao2009}, for example, describes a procedure to solve \eqref{eq:w_spto} under the following three conditions: i) $f$ is quadratic, ii) the uncertainty is only present in the coefficients of the linear terms of $f$, and iii) no constraints are imposed on the decision $z$, which means $Z = \mathbb{R}^n$. Some years later, the authors \cite{Donti2017} proposed a heuristic gradient-based procedure to solve \eqref{eq:w_spto} for strongly convex problems with deterministic equality constraints and inequality chance constraints. Almost concurrently, reference \cite{Elmachtoub2017} discusses the difficulties of solving \eqref{eq:w_spto} in the case of linear problems, since such a formulation may lead to an uninformative loss function. To overcome this issue, they successfully develop a convex surrogate that allows to efficiently train $g^{\text {SP}}(x_{i};w)$ in the linear case.  More recently, the authors in \cite{Wilder2019} suggest a similar approach as in \cite{Donti2017} to combinatorial problems with a regularized linear objective function.

In summary, the four references above propose ad-hoc gradient methods for specific instances of \eqref{eq:w_spto}. However, the technical literature lacks, to the best of our knowledge, a general-purpose procedure to solve such a problem using available optimization solvers. To fill this gap, we propose the following bilevel program \cite{Dempe2017} as a generic mathematical formulation of \eqref{eq:w_spto}:
\begin{subequations}
\begin{align}
    w^{\text {BL}} \in \arg \min_{w \in \mathbb{R}^q{\color{black};\,\hat{z}_{i}}  } &\ \sum_{i \in \mathcal{N}} f(\hat{z}_{i}; y_{i}) \label{eq:spto_upper}\\
    \text{s.t.} & \enskip \hat{z}_{i} \in \arg \min_{z \in Z} f(z; g^{\text {BL}}(x_{i}; w)), \enskip \forall i \in \mathcal{N} \label{eq:spto_lower}
\end{align}\label{eq:bilevel_framework}
\end{subequations}
\noindent \textcolor{black}{where $g^{\text {BL}}: \mathcal{X} \times \mathbb{R}^q \rightarrow \mathbb{R}^m$ is selected similarly to $g^{\text{FO}}$ and $g^{\text{SP}}$.} Problem \eqref{eq:bilevel_framework} is formulated as a bilevel optimization model commonly used to mathematically characterize non-cooperative and sequential Stackelberg games in which the \textit{leader} makes her decisions anticipating the reaction of the \textit{follower} \cite{Labbe2019}. In this sense, the upper-level problem determines the optimal parameter vector $w$ anticipating the decision provided by each lower-level problem \eqref{eq:spto_lower} if the value $\hat{y}_i$ is given by $g^{\text {BL}}(x_{i};w)$. We denote this approach based on bilevel programming as BL (acronym of BiLevel). In Section \ref{subsec:solution_strategy}, we discuss the assumptions that problem \eqref{eq:ms0_obj} must satisfy so that problem \eqref{eq:bilevel_framework} can be reformulated as a single-level optimization problem to be solved using off-the-shelf optimization solvers. \textcolor{black}{Although solving the bilevel problem \eqref{eq:bilevel_framework} may be computationally expensive, this is a task that can be performed offline.} Once $w^{\text {BL}}$ is determined, the optimal decision $z^{\text {BL}}$ under context $X=x$ is computed by solving the following problem:
\begin{align}
    z^{\text {BL}}(x) \in \arg \min_{z \in Z} & \enskip f(z; g^{\text {BL}}(x;w^{\text {BL}})) \label{eq:determ_bl}
\end{align}

\textcolor{black}{The bilevel program \eqref{eq:spto_upper}--\eqref{eq:spto_lower} computes the value for the parameter vector $w$ that maximizes the \emph{in-sample} performance of the surrogate decision-making model~\eqref{eq:determ_bl}. For this estimation to be of use, it must be guaranteed that under two contexts $x_i$, $x_i'$, such that $x_i = x_i'$, it holds $\hat{z}_i = \hat{z}_i'$, i.e., under equal contexts, equal decisions. This is a condition that is reminiscent of the notion of \emph{non-anticipativity} in Stochastic Programming. Importantly, this condition is automatically satisfied if the solution to the lower-level problem~\eqref{eq:spto_lower} is unique for any value of $w$. Otherwise, the bilevel program \eqref{eq:spto_upper}--\eqref{eq:spto_lower} would choose the $\hat{z}_i$ from the optimal solution set of \eqref{eq:spto_lower} that minimizes the upper-level objective function \eqref{eq:spto_upper} \emph{given} ---i.e., by anticipating--- the uncertainty outcome $y_i$. This is so because the bilevel program \eqref{eq:spto_upper}--\eqref{eq:spto_lower}, as we have formulated it, delivers the optimistic Stackelberg solution.}
For instance, let us assume that there exists a value $\tilde{w}$ such that $f(z;g^{\rm BL}(x_i;\tilde{w})) = \vartheta$ for all $i \in \mathcal{N}$, where $\vartheta$ is a constant. In this case, the lower-levels \eqref{eq:spto_lower} boil down to feasibility problems imposing that $z\in Z$ and therefore, $\hat{z}_i$ can violate non-anticipativity and adapt to realization $y_i$ for all $i \in \mathcal{N}$. More importantly, using $\tilde{w}$ in \eqref{eq:determ_bl} would lead to degenerate and highly suboptimal decisions under any context $X=x$. This issue is reported in \cite{Elmachtoub2017} for linear objective functions, where authors propose a convex surrogate function of $l^{\text{SP}}$ to train meaningful instances of model ${g^{\text{SP}}(\cdot; w^{\text{SP}})}$. Similarly, we propose in Section \ref{sec:strategic}  a modified lower-level surrogate model for the strategic producer problem in order to ensure non-anticipativity for any parameter vector $w$.




Next, we discuss other surrogate decision-making models different from \eqref{eq:surrogate}, which have also been recently proposed to approximate the solution of \eqref{eq:ms0_obj}. For this purpose, notice first that problem \eqref{eq:ms0_obj} can be equivalently recast as
\begin{align}
    \min_{z \in Z} & \enskip \mathbb{E} [f(z; Y) | X = x] = \min_{z \in Z}  \enskip \mathbb{E}_{\mathcal{Q}_{|x}} [f(z; Y)] \label{eq:ms0_obj_bis} 
\end{align}
where $\mathcal{Q}_{|x}$ represents the conditional probability distribution of $Y$ given $X=x$. Thus, a second family of surrogate decision-making models can be introduced with the following general form:
\begin{align}
    \min_{z \in Z}  \enskip \mathbb{E}_{\widehat{\mathcal{Q}}_{|x}} [f(z; Y)] \label{eq:surrogate_Bert} 
\end{align}
where $\widehat{\mathcal{Q}}_{|x}$ is an approximation of the unknown probability measure $\mathcal{Q}_{|x}$ that is constructed from the available sample $S = \left\lbrace (y_{i}, x_{i}), \forall i \in \mathcal{N} \right\rbrace$. For the surrogate problem~\eqref{eq:surrogate_Bert} to be computationally tractable, the proxy $\widehat{\mathcal{Q}}_{|x}$ is often built as a discrete probability distribution supported on a finite number of points, more specifically, on the $y$-locations of the sample, i.e., $\{y_{i}, \forall i \in \mathcal{N}\}$. This way, the solution to~\eqref{eq:surrogate_Bert} under context $X = x$, which we denote as $z^{\text {ML}}(x)$, can be generically expressed as:
\begin{align}
    z^{\text {ML}}(x) \in \arg \min_{z \in Z} & \enskip \sum_{i \in \mathcal{N}} g^{\text {ML}}(x, x_{i}; w) f(z; y_{i})\label{SAA}
\end{align}
with $\{g^{\text {ML}}(x, x_{i}; w), \forall i \in \mathcal{N}\}$ being the probability masses that the specific proxy $\widehat{\mathcal{Q}}_{|x}$ that is used places on $\{y_{i}, \forall i \in \mathcal{N}\}$. These masses or weights are determined as a function $g^{\text {ML}}: \mathcal{X} \times \mathcal{X} \times \mathbb{R}^q \rightarrow \mathbb{R}$ of the historical contextual information $x_{i}$, the current context $x$, and some parameters $w$.



In essence, this scheme adapts the Sample Average Approximation (a well-known data-driven solution strategy in Stochastic Programming \cite{kleywegt2002sample, ruszczynski2003stochastic}) to the case of \emph{conditional} stochastic programs. It was first formalized in \cite{Bertsimas2020} and, since then, has been subject to a number of improvements (e.g., regularization procedures for bias-variance reduction \cite{Ho2019}; robustification \cite{Bertsimas2017}; and algorithmic upgrades \cite{Diao2020}) and extensions, e.g., to a dynamic decision-making setting \cite{Bertsimas2019dynamic}. Recently, the work in \cite{Mundru2019} introduces a bilevel formulation to optimally tune the parameters $w$ that determine the weights $g^{\text {ML}}(x,x_{i};w)$. Using our notation, the method proposed in \cite{Mundru2019} can be formulated as follows:
\begin{subequations}
\begin{align}
    w^{\text {ML}} \in \arg \min_{w \in \mathbb{R}^q{\color{black};\, \hat{z}_{i}}} & \ \sum_{i \in \mathcal{N}} f(\hat{z}_{i}; y_{i}) \label{eq:b_theor_upper}\\
    \text{s.t.} & \enskip \hat{z}_{i} \in \arg \min_{z \in Z} \sum_{i'\in \mathcal{N}:i'\neq i} g^{\text {ML}}(x_{i}, x_{i'}; w) f(z; y_{i'}), \enskip \forall i \in \mathcal{N} \label{eq:b_theor_lower}
\end{align} \label{eq:b_theor}
\end{subequations}

\noindent where the function $g^{\text {ML}}: \mathcal{X} \times\mathcal{X} \times \mathbb{R}^q \rightarrow \mathbb{R}$ used to compute the weights can be chosen from a catalog of several classical machine learning algorithms $G^{\text {ML}}$ such as $k$-nearest neighbors, Nadaraya-Watson kernel regression or Random Forest. \textcolor{black}{The author of \cite{Mundru2019} resorts} to tailor-made approximations and greedy algorithms for each machine learning technique that is used to construct function $g^{\text {ML}}$, but do not provide a general-purpose solution strategy valid for any function $g^{\text {ML}}$. This approach, which is based on machine learning techniques, is called ML (stands for Machine Learning). After solving \eqref{eq:b_theor}, the optimal decision $z^{\text {ML}}(x)$ under context $X = x$ is obtained by solving \eqref{SAA} with $w=w^{\text {ML}}$.

\label{paragraph:rev1_R1.3_start} The surrogate problems~\eqref{eq:surrogate} and~\eqref{eq:surrogate_Bert} are, by design, different, in part because they are the result of distinct frameworks to address the conditional stochastic program~\eqref{eq:ms0_obj}. The surrogate problem~\eqref{eq:surrogate} is based on the assumption that it is possible to find a good decision $z$ in terms of the conditional expected cost $\mathbb{E} [f(z; Y) | X = x]$ by optimizing that decision for a single scenario $\hat{y}$ of the uncertainty $Y$. Naturally, all the complexity of this approach lies in how to infer, from the data sample $S$, the single scenario $\hat{y}$ that unlocks the best decision $z$. This inference process makes use of global methods that consider all data points in the sample to obtain more robust decision mappings. In contrast, all the difficulty of the surrogate problem~\eqref{eq:surrogate_Bert} rests on how to retrieve a good approximation of the true conditional distribution $\mathcal{Q}_{|x}$ from the sample $S$. Such an approximation is performed using local machine learning methods that only employ data close to the given context $x$ and consequently, a large amount of data is required to avoid overfitting. In more practical terms, embedding local machine learning methods into the estimation problem \eqref{eq:b_theor} makes this problem computationally intractable in most cases. Besides, the surrogate problem~\eqref{eq:surrogate} is computationally less demanding than \eqref{eq:surrogate_Bert}, because the latter requires evaluating the cost function $f(z; \cdot)$ for multiple values of the uncertainty $Y$. \label{paragraph:rev1_R1.3_end}

Finally, there is a third class of surrogate decision-making models that arises from the idea of using the sample $S$ to directly learn the optimal decision $z$ as a function of the context $x$, this way bypassing the need for constructing the estimate $\hat{y}$ or the proxy distribution $\widehat{\mathcal{Q}}_{|x}$. Following this logic, we seek a decision rule or mapping $g^{\text {DR}}: \mathcal{X} \times \mathbb{R}^q \rightarrow \mathbb{R}^n$ from a family $G^{\text {DR}}$ so that $\arg \min_{z \in Z} \mathbb{E}[f(z;Y)|X=x] \approx \hat{z} = g^{\text {DR}}(x;w)$. Particularizing for the empirical distribution of the data, this approach renders:
\begin{subequations}
\begin{align}
    w^{\text {DR}} \in \arg \min_{w \in \mathbb{R}^q} & \sum_{i \in \mathcal{N}} f(g^{\text {DR}}(x_{i};w); y_{i}) \label{eq:approach_rule0} \\
    \text{s.t.} & \enskip g^{\text {DR}}(x_{i};w) \in Z, \enskip \forall i \in \mathcal{N} \label{eq:approach_rule_1}
\end{align} \label{eq:decision_rule_framework}
\end{subequations}

One clear advantage of directly learning the optimal decision policy is that, after solving \eqref{eq:decision_rule_framework}, the decision $z^{\rm DR}$ to be implemented under context $X = x$ is efficiently computed as follows:
\begin{align}
    z^{\text {DR}}(x) = g^{\rm DR}(x;w^{\text {DR}}) \label{eq:deci_rule}
\end{align}
Actually, the mapping~\eqref{eq:deci_rule} constitutes the surrogate decision-making model itself. This method, which aims at determining an optimal decision rule, is denoted as DR (acronym of Decision Rule). Nevertheless, feasibility issues may arise as this approach does not necessarily guarantee that the resulting $z^{\text {DR}}$ obtained through \eqref{eq:deci_rule} belongs to $Z$ for any plausible context $x$. The authors of \cite{Ban2019} propose and investigate this approach for the popular newsvendor problem, for which they consider a linear decision rule. Their newsvendor formulation does not involve any constraint and therefore, decisions yielded by \eqref{eq:deci_rule} are always valid. However, the use of this approach is questionable for many other practical applications in which decisions must satisfy a set of constraints.  

In summary, the contributions of the proposed bilevel model \eqref{eq:bilevel_framework} with respect to the other approaches presented in this section are:
\begin{itemize}
    \item[-] Unlike the traditional approach \eqref{eq:forecasting_step_theo}, ours provides estimations of $y$ by leveraging information about the optimization problem to be solved.
    \item[-] \label{paragraph:rev1_R3.13} Unlike the existing ``predict-then-optimize'' methodology \eqref{eq:w_spto}, our approach is formulated as a generic bilevel optimization that, under convexity assumptions, is reformulated as a single-level optimization problem that can be solved using off-the-shelf optimization software.
    \item[-] Unlike approach \eqref{eq:b_theor}, ours makes uses of global estimation methods that use all available data to infer the point forecast of the uncertainty that unlocks the best decision. Therefore, our approach is less prone to overfitting, especially for small data samples. In addition, formulation \eqref{eq:b_theor} is more difficult to solve than \eqref{eq:bilevel_framework}.
    \item[-] Unlike approach  \eqref{eq:decision_rule_framework}, ours guarantees the feasibility of the resulting optimal decision under any context.
\end{itemize}

\section{Solution strategy} \label{subsec:solution_strategy}

In this section, we elaborate on how to solve the general-purpose bilevel program~\eqref{eq:bilevel_framework} we propose to compute the best single scenario $\hat{y}$ to be fed into the surrogate problem \eqref{eq:surrogate}. To do so, we particularize the generic formulation \eqref{eq:ms0_obj} as follows:
\begin{subequations}
\begin{align}
    \min_{z,s(Y)} & \enskip \mathbb{E} [f_0(z,s(Y); Y) | X = x] \label{eq:stochastic_two_stages_obj} \\
    \text{s.t.} & \enskip f_{j}(z,s(Y);Y) \le 0, \enskip \forall j \in J \label{eq:inequality} \\
    & \enskip h_{k}(z,s(Y);Y) = 0,  \enskip \forall k \in K \label{eq:equality}
\end{align} \label{eq:stochastic_two_stages}
\end{subequations}
\noindent where $z$ constitutes the vector of \emph{here-and-now} variables independent of the uncertainty, $s(Y)$ represents the \emph{wait-and-see} decisions, and constraints \eqref{eq:inequality}, \eqref{eq:equality} must be satisfied for $\mathcal{Q}_{|x}$-almost all $y$ (i.e., with probability one). We also assume that $f_{0},f_{j}$ are convex functions with respect to all variables, $h_k$ are affine functions, and function $g^{\text {BL}}$ is continuous in the parameter vector $w$. 

Our method solves the following bilevel optimization problem:
\begin{subequations}
\begin{align}
    w^{\text {BL}} \in & \arg \min_{w \in \mathbb{R}^q {\color{black};\, \hat{z}_i}} \  \sum_{i \in \mathcal{N}} f_0(\hat{z}_i,\hat{s}_i;y_i) \\
    & \text{s.t.} \enskip f_{j}(\hat{z}_i,\hat{s}_i;y_i) \le 0, \enskip \forall j \in J, \forall i \in \mathcal{N} \\
    & \phantom{\text{s.t.}} \enskip h_{k}(\hat{z}_i,\hat{s}_i;y_i) = 0, \enskip \forall k \in K, \forall i \in \mathcal{N} \\
    & \phantom{\text{s.t.}} \enskip \hat{z}_{i} \in \{ \arg \min_{z,s} \enskip f_0(z,s;g^{\text {BL}}(x_{i}; w)) \label{eq:bilevel_two_stage_lower_fo}\\
    & \hspace{18mm} \text{s.t.} \enskip f_{j}(z,s;g^{\text {BL}}(x_{i}; w)) \le 0, \enskip \forall j \in J \label{eq:bilevel_two_stage_lower_ineq}\\
    & \hspace{25mm} h_{k}(z,s;g^{\text {BL}}(x_{i}; w)) = 0, \enskip \forall k \in K \}, \forall i \in \mathcal{N} \label{eq:bilevel_two_stage_lower_eq}
\end{align}\label{eq:bilevel_two_stage}
\end{subequations}

On the assumption that the lower-level minimization problems \eqref{eq:bilevel_two_stage_lower_fo}--\eqref{eq:bilevel_two_stage_lower_eq} satisfy some constraint qualification, the classical strategy to solve \eqref{eq:bilevel_two_stage} is to replace each lower level \eqref{eq:bilevel_two_stage_lower_fo}--\eqref{eq:bilevel_two_stage_lower_eq} with its equivalent Karush-Kuhn-Tucker (KKT) conditions \cite{Boyd2004}, that is,
\begin{subequations}
\begin{align}
    w^{\text {BL}} \in & \arg \min_{w \in \mathbb{R}^q {\color{black}; \, \hat{z}_i, \lambda_{ji}}}\ \sum_{i \in \mathcal{N}} f_0(\hat{z}_i,\hat{s}_i;y_i) \label{eq:kkt_obj}\\
    &\text{s.t.} \enskip f_{j}(\hat{z}_i,\hat{s}_i;y_i) \le 0, \enskip \forall j \in J, \forall i \in \mathcal{N} \label{eq:kkt_ineq} \\
    & \phantom{\text{s.t.}} \enskip h_{k}(\hat{z}_i,\hat{s}_i;y_i) = 0, \enskip \forall k \in K, \forall i \in \mathcal{N} \label{eq:kkt_eq}\\
    & \phantom{\text{s.t.}} \enskip \nabla f_0(\hat{z}_{i},s_i; g^{\text {BL}}(x_{i}; w)) + \sum_{j\in J} \lambda_{ji} \nabla f_{j}(\hat{z}_{i},s_i;g^{\text {BL}}(x_{i};w)) + \nonumber \\
    &\phantom{s.t.} \qquad + \sum_{k\in K} \upsilon_{ki} \nabla h_{k}(\hat{z}_{i},s_i;g^{\text {BL}}(x_{i};w)) = 0, \enskip \forall i \in \mathcal{N} \label{eq:stationarity} \\
    &\phantom{s.t.} \enskip f_{j}(\hat{z}_{i},s_i;g^{\text {BL}}(x_{i};w)) \le 0,  \enskip \forall j\in J, \enskip \forall i \in \mathcal{N} \label{eq:primal_feas_f} \\
    &\phantom{s.t.} \enskip h_{k}(\hat{z}_{i},s_i;g^{\text {BL}}(x_{i};w)) = 0,  \enskip \forall k \in K, \enskip \forall i \in \mathcal{N}  \label{eq:primal_feas_h} \\
    &\phantom{s.t.} \enskip \lambda_{ji} \ge 0, \enskip \forall j \in J, \enskip \forall i \in \mathcal{N} \label{eq:dual_feas} \\
    &\phantom{s.t.} \enskip \lambda_{ji} f_{j}(\hat{z}_{i},s_i;g^{\text {BL}}(x_{i};w)) = 0, \enskip \forall j \in J, \enskip \forall i \in \mathcal{N} \label{eq:slackness} 
    \end{align} \label{eq:kkt_reformulation}
\end{subequations}

where $\lambda_{ji}, \upsilon_{ki} \in \mathbb{R}$ are, respectively, the Lagrange multipliers related to constraints \eqref{eq:bilevel_two_stage_lower_ineq} and \eqref{eq:bilevel_two_stage_lower_eq} for each lower-level problem; \eqref{eq:kkt_obj} and \eqref{eq:kkt_ineq}-\eqref{eq:kkt_eq} are, in that order,  the objective function and constraints of the upper-level problem, and constrains \eqref{eq:stationarity}, \eqref{eq:primal_feas_f}-\eqref{eq:primal_feas_h}, \eqref{eq:dual_feas}, \eqref{eq:slackness}, are, respectively, the stationarity, primal feasibility, dual feasibility and slackness conditions of the lower-level problems. As discussed in \cite{Scheel2000}, problem \eqref{eq:kkt_reformulation} violates the Mangasarian-Fromovitz constraint qualification at every feasible point and therefore, interior-point methods fails to find even a local optimal solution to this problem. To overcome this issue, a regularization approach was first introduced in \cite{Scholtes2001} and further investigated in \cite{Ralph2004}. This method replaces all complementarity constraints \eqref{eq:slackness} with inequality \eqref{eq:slackness2} below: 
\begin{subequations}
\begin{align}
    w^{\text {BL}} \in & \arg \min_{w \in \mathbb{R}^q{\color{black}; \, \hat{z}_i, \lambda_{ji}}}\  \sum_{i \in \mathcal{N}} f_0(\hat{z}_i,\hat{s}_i;y_i) \label{eq:kkt_obj_reg}\\
    &\text{s.t.} \enskip \eqref{eq:kkt_ineq}- \eqref{eq:dual_feas}\\
    &\phantom{s.t.} \enskip - \sum_{j\in J} \lambda_{ji} f_{j}(\hat{z}_{i},s_i;g^{\text {BL}}(x_{i};w)) \leq \epsilon, \enskip \forall i \in \mathcal{N} \label{eq:slackness2} 
    \end{align} \label{eq:kkt_reformulation_reg}
\end{subequations}

\noindent where $\epsilon$ is a small non-negative scalar that allows to reformulate \eqref{eq:kkt_reformulation} as the parametrized nonlinear optimization problem \eqref{eq:kkt_reformulation_reg}, which typically satisfies a constraint qualification and can be then efficiently solved by standard non-linear optimization solvers. Authors of \cite{Scholtes2001} prove that, as $\epsilon$ tends to 0, the solution of \eqref{eq:kkt_reformulation_reg} tends to a \emph{local} optimal solution of \eqref{eq:kkt_reformulation}. In the remaining of the manuscript, we will refer to this approach as BL-R.

Alternatively, the complementarity slackness conditions can be linearized according to Fortuny-Amat \cite{Fortuny1981} as follows:
\begin{subequations}
\begin{align}
    w^{\text {BL}} \in & \arg \min_{w \in \mathbb{R}^q{\color{black}; \, \hat{z}_i, \lambda_{ji}, u_{ji}}}  \ \sum_{i \in \mathcal{N}} f_0(\hat{z}_{i},\hat{s}_i; y_{i}) \label{eq:kkt_obj_fa}\\
    &\text{s.t.} \enskip \eqref{eq:kkt_ineq}- \eqref{eq:dual_feas}\\
    &\phantom{s.t.} \enskip \lambda_{ji} \le u_{ji} M^{D}, \enskip \forall j\in J, \enskip \forall i \in \mathcal{N}\\
    &\phantom{s.t.} \enskip f_{j}(\hat{z}_{i},s_i;g^{\text {BL}}(x_{i};w)) \ge ( u_{ji} - 1 ) M^{P}, \enskip \forall j\in J, \enskip \forall i \in \mathcal{N} \label{eq:amat_dual2} \\
    &\phantom{s.t.} \enskip u_{ji} \in \{0, 1\}, \enskip \forall j\in J, \enskip \forall i \in \mathcal{N}  \label{eq:amat_binary}
\end{align} \label{eq:kkt_reformulation_fa}
\end{subequations}

\noindent where $u_{ji}$ are binary variables, and $M^{P}, M^{D} \in \mathbb{R}^+$ are large enough constants whose values can be determined as proposed in \cite{Pineda2018}. The resulting model \eqref{eq:kkt_reformulation_fa} is a single-level mixed-integer non-linear problem. We denote this method as BL-M.

Solving the bilevel problem \eqref{eq:bilevel_framework} using either BL-R or BL-M is valid for a conditional stochastic problem that satisfies the conditions described in this section. Nonetheless, the complexity of solving the regularized non-linear problem \eqref{eq:kkt_reformulation_reg} or the mixed-integer non-linear program \eqref{eq:kkt_reformulation_fa}  highly depends on functions $f_0,f_j,h_k,g^{\text {BL}}$. In some cases (see, for instance, the particular applications discussed in Section~\ref{sec:application}), problem \eqref{eq:kkt_reformulation_fa} can be reformulated as a mixed-integer linear/quadratic optimization problem that can be solved to global optimality using standard optimization solvers. In the general case, problems \eqref{eq:kkt_reformulation_reg} and \eqref{eq:kkt_reformulation_fa} can also be  solved using off-the-shelf optimization solvers, but global optimality may not be guaranteed. Notwithstanding this, local optimal solutions of the proposed bilevel formulation \eqref{eq:bilevel_framework} may still lead to optimal decisions that are significantly better than those computed by FO or DR.

\section{Applications}
\label{sec:application}

In Section \ref{sec:math_framework}, we introduce a common  mathematical framework to present five different approaches for contextual decision-making under uncertainty, namely, the \emph{predict-then-optimize} strategies FO, SP, and BL; method ML, which relies on a proxy of the true conditional distribution that is built using machine-learning techniques, and the decision-rule approach DR. Unfortunately, in the technical literature, methods SP and ML have only been applied to conditional stochastic optimization problems with a specific structure and they both lack a solution strategy for more general conditional stochastic programs.  For this reason, in this section, we limit ourselves to comparing  approaches FO, BL, and DR on various contextual decision-making problems under uncertainty, each of which illustrates a certain relevant aspect of our proposal. Subsection \ref{sec:newsvendor} compares these methodologies using the newsvendor problem, a well-known stochastic programming problem with simple recourse. The proposed methodology is also applied in Subsection~\ref{sec:placement} to the product placement problem, a two-stage stochastic programming problem with full recourse. Finally, Subsection \ref{sec:strategic} presents a strategic producer problem formulated as a one-stage stochastic programming in which the uncertainty only affects the objective function.

\subsection{Newsvendor problem}\label{sec:newsvendor}

We start with the popular newsvendor problem in the spirit of \cite{Ban2019}, a work that elicited renewed interest \cite{Mundru2019, bertsimas2018} in the solution to the conditional stochastic program~\eqref{eq:ms0_obj}. In the newsvendor problem, the goal of the decision maker is to find the optimal ordering quantity for a product with unknown random demand $Y$. In turn, this (positive) demand may be influenced by a random vector of features $X$ representing, for instance, product information, weather conditions, customer profiles, etc. The decision maker has, therefore, a collection of observations $\{(x_i,y_i), \forall i \in \mathcal{N}\}$, which s/he would like to exploit to make an informed ordering quantity $z$ under the context $X=x$. Let $d$ and $r$, with $r > d> 0$, be the cost and revenue of manufacturing and selling one product unit, respectively. This problem can be formulated as the following conditional stochastic program:
\begin{align}
    \min_{z \in \mathbb{R}} & \enskip \mathbb{E} [dz - r\min(z,Y) | X = x] \label{eq:newsvendor} 
\end{align}

Approaches FO and BL both follow a ``predict-then-optimize'' strategy, whereby the ordering quantity is obtained as the solution to the following surrogate decision-making model:
\begin{align}
    \min_{z \in \mathbb{R}} & \enskip dz - r\min(z,\hat{y}) \label{eq:newsvendor_surrogate} 
\end{align}
We can use an auxiliary variable $s$ to get rid of the inner minimization and write instead:
\begin{subequations}
\begin{align}
    \min_{z, s} & \enskip dz - rs \\
    \text{s.t.} & \enskip s \leq z \\
    \phantom{\text{s.t.}} & \enskip s \leq \hat{y}
\end{align} 
\label{eq:newsvendor_surrogate_ref}
\end{subequations}
whose solution is trivial, namely, $z^*=s^*=\hat{y}$.

FO and BL differ in the particular single value or scenario $\hat{y}$ that each of them uses. In the case of FO, $\hat{y}$ is an estimate of $\mathbb{E}[Y|X=x]$. Consequently, it becomes apparent that, for the newsvendor problem, approach FO is fundamentally inconsistent, because it is well-known that the solution to~\eqref{eq:newsvendor} corresponds to the quantile $\frac{r-d}{r}$ of the demand distribution $Y$ conditional on $X=x$. Naturally, this quantile is generally different from $\mathbb{E}[Y|X=x]$.

Now, if we take $\hat{y} = g^{\text {BL}}(x;w) = w^T x$ in our approach, the optimal vector of linear coefficients $w^{\rm BL}$ is computed as follows:
\begin{subequations}
\begin{align}
    w^{\rm BL} \in \arg \min_{w\in\mathbb{R}^{p}} & \enskip \sum_{i \in \mathcal{N}} d \hat{z}_i - r\min({\hat{z}_i,y_i}) \\
    \text{s.t.} & \enskip \hat{z}_{i} \in \{ \arg \min_{z_i,s_i} \enskip dz_i - rs_i  \\
    & \qquad\qquad\qquad \text{s.t.} \enskip s_i \leq z_i \\
    & \qquad\qquad\qquad \phantom{\text{s.t.}} \enskip s_i \leq w^T x_i \}, \forall i \in \mathcal{N}
\end{align} 
\label{eq:newsvendor_bilevel}
\end{subequations}
which, based on our previous argument, boils down to:
\begin{subequations}
\begin{align}
    w^{\rm BL} \in \arg \min_{w\in\mathbb{R}^{p}} & \enskip \sum_{i \in \mathcal{N}} d \hat{z}_i - r\min({\hat{z}_i,y_i}) \\
    \text{s.t.} & \enskip \hat{z}_{i} = w^T x_i, \enskip \forall i \in \mathcal{N}
\end{align} 
\label{eq:newsvendor_bilevel_reduced}
\end{subequations}
Therefore, our approach coincides exactly with that proposed in~\cite{Ban2019}, which, in turn, is given by problem~\eqref{eq:decision_rule_framework} in Section~\ref{sec:math_framework} when $g^{\text {DR}}(x;w) = w^T x$. This equivalence is far from being general though, as we will see with the other applications below.

\subsection{Product placement problem}\label{sec:placement}
Given a graph $\mathcal{G} = (\mathcal{B}, \mathcal{A})$ with node-arc matrix $A$, in the product placement problem, the goal is to decide the amount $z_b \in \mathbb{R}^{+}$ of a certain product to be placed in each node $b \in \mathcal{B}$ of the grid \cite{Bertsimas2020}. After this decision is made, the demand for the product at each node $y_b$ is realized, and the inventories of product throughout the network are shipped across the arcs $\mathcal{A}$ so as to satisfy the actually observed nodal demands. As in the newsvendor problem, these demands may be affected by some exogenous factors $X$ that may be also random, but that are disclosed before the product placement decision is to be made. Let $h \in \mathbb{R}^{|\mathcal{B}|}$ and $g \in \mathbb{R}^{|\mathcal{A}|}$ be the cost of initially placing products in the nodes of the network and the cost of shipping products through the edges of the graph, respectively. The product placement problem under uncertain demand, but with contextual information, can be formulated as follows:
\begin{align}
    \min_{z \geq 0} & \enskip \mathbb{E} [c(z; Y) | X = x] \label{eq:placement} 
\end{align}
where
\begin{subequations}
\begin{align}
    c(z; y) = h^Tz + & \min_{f\geq 0,p\geq 0} \enskip g^Tf + r^Tp \label{eq:placement2_obj} \\
    & \text{s.t.} \enskip Af \leq z - y + p \label{eq:placement2_con}
\end{align}\label{eq:placement2} 
\end{subequations}
In problem~\eqref{eq:placement2}, we have included a variable vector $p \in \mathbb{R}_{\geq 0}^{|\mathcal{B}|}$ to allow for unsatisfied demand, with the associated penalty cost $r \in \mathbb{R}^{|\mathcal{B}|}$. Furthermore, the decision vector $f \in \mathbb{R}^{|\mathcal{A}|}$ represents the amount of product shipped across the arcs of the network. The cost function~\eqref{eq:placement2_obj} takes the form of a two-stage linear cost, with the integration of a recourse problem. More importantly, unlike in the newsvendor problem, the recourse is given by a full-fledged (linear) minimization problem. The surrogate decision-making model associated with the \emph{predict-then-optimize} strategies FO and BL is as follows:
\begin{subequations}
\begin{align}
    \min_{z \geq 0, f\geq0, p\geq0} & \enskip h^Tz +  g^Tf + r^Tp  \\
    & \text{s.t.} \enskip Af \leq z - \hat{y} + p 
\end{align}\label{eq:placement_surrogate} 
\end{subequations}
To ease the exposition and the notation that follows, we make the additional assumption that $r > h >0$, where the inequality holds component-wise. In this case, variable vector $p$ in \eqref{eq:placement_surrogate} is zero at the optimum and the surrogate model can be simplified to:
\begin{subequations}
\begin{align}
    \min_{z \geq 0, f\geq0} & \enskip h^Tz +  g^Tf  \\
    & \text{s.t.} \enskip Af \leq z - \hat{y} 
\end{align}\label{eq:placement_surrogate_simplified} 
\end{subequations}

As previously discussed, problem~\eqref{eq:placement_surrogate_simplified} is a deterministic mathematical program whereby the decision $z$ is solely optimized for the point prediction of demand $\hat{y}$. While the traditional FO approach sets such a prediction to $\mathbb{E}[Y|X=x]$, the rationale behind the approach BL is to compute a $W$-parameterized function such that the surrogate problem~\eqref{eq:placement_surrogate_simplified} delivers the decision $z$ that minimizes the in-sample cost, that is:  
\begin{subequations}
\begin{align}
    W^{\rm BL} \in \arg \min_{W\in\mathbb{R}^{|\mathcal{B}|\times p}} & \enskip \sum_{i \in \mathcal{N}} h^T\hat{z}_i +  g^T\hat{f}_i + r^T\hat{p}_i  \\
    \text{s.t.} & \enskip A\hat{f}_i \leq \hat{z}_i - y_i + \hat{p}_i, \enskip \forall i \in \mathcal{N}  \\
    \phantom{\text{s.t.}} & \enskip \hat{f}_i, \hat{p}_i \geq 0, \enskip \forall i \in \mathcal{N}\\
    \phantom{\text{s.t.}} & \enskip \hat{z}_i \in \{ \arg \min_{z_i \geq 0, f_i\geq0} \enskip h^Tz_i +  g^Tf_i  \label{eq:placement_bilevel_obj}\\
    & \hspace{21mm} \text{s.t.} \enskip Af_i \leq z_i - Wx_i \}, \enskip \forall i \in \mathcal{N}  \label{eq:placement_bilevel_con} 
\end{align} \label{eq:placement_bilevel} 
\end{subequations}
where we have taken $\hat{y} = g^{\rm BL}(x; W) = W x$ with $W \in \mathbb{R}^{|\mathcal{B}| \times p}$. \textcolor{black}{As discussed in Section \ref{sec:math_framework}, the lower-level problem \eqref{eq:placement_bilevel_obj}--\eqref{eq:placement_bilevel_con} must have a unique solution. This can be guaranteed if, for example, all the shipping routes that can be taken to satisfy each demand in the graph entail a different cost. If this condition is not satisfied, the degeneracy of the lower-level problem can be eliminated by using classical results from the linear programming literature as described in \cite{Garcia2021}.} As stated in Section~\ref{subsec:solution_strategy},  the solution to \eqref{eq:placement_bilevel} can be addressed by replacing the lower-level linear program \eqref{eq:placement_bilevel_obj}--\eqref{eq:placement_bilevel_con} with its KKT optimality conditions: 
\begin{subequations}
\begin{align}
    W^{\rm BL} \in \arg \min_{W\in\mathbb{R}^{|\mathcal{B}|\times p}} & \enskip \sum_{i \in \mathcal{N}} h^T\hat{z}_i +  g^T\hat{f}_i + r^T\hat{p}_i  \\
    \text{s.t.} & \enskip A\hat{f}_i \leq \hat{z}_i - y_i + \hat{p}_i, \enskip \forall i \in \mathcal{N} \\
    \phantom{\text{s.t.}} & \enskip \hat{f}_i, \hat{p}_i \geq 0, \enskip \forall i \in \mathcal{N}\\
    \phantom{\text{s.t.}} & \enskip 0 \leq (h-\alpha_i) \perp \hat{z}_i \geq 0, \enskip \forall i \in \mathcal{N} \label{eq:placement_single_level_com1}\\
    \phantom{\text{s.t.}} & \enskip 0 \leq (g+A^T\alpha_i) \perp f_i \geq 0, \enskip \forall i \in \mathcal{N} \label{eq:placement_single_level_com2}\\
    \phantom{\text{s.t.}} & \enskip 0 \leq (\hat{z}_i-Af_i-Wx_i) \perp \alpha_i \geq 0, \enskip \forall i \in \mathcal{N} \label{eq:placement_single_level_com3}
\end{align} \label{eq:placement_single_level} 
\end{subequations}

where $\alpha_i \in \mathbb{R}^{|\mathcal{B}|}$ is the vector of Lagrange multipliers associated with constraint \eqref{eq:placement_bilevel_con}. Thus, problem \eqref{eq:placement_single_level} can be solved by regularizing the complementary slackness conditions or by using their Fortuny-Amat big-M reformulation. In the latter case, we arrive to a MIP problem that can be solved using commercial optimization software such as CPLEX or GUROBI.

Finally, if we also take a linear decision mapping $z(x) = g^{\rm DR}(x;W) = W x$ where $W \in \mathbb{R}^{|\mathcal{B}| \times p}$, the DR approach solves the following minimization problem to compute the optimal matrix of linear coefficients $W^{\rm DR}$:
\begin{subequations}
\begin{align}
    W^{\rm DR} \in \arg \min_{W\in\mathbb{R}^{|\mathcal{B}|\times p}} & \enskip \sum_{i \in \mathcal{N}} h^T\hat{z}_i +  g^T\hat{f}_i + r^T\hat{p}_i  \\
    \text{s.t.} & \enskip A\hat{f}_i \leq \hat{z}_i - y_i + \hat{p}_i, \enskip \forall i \in \mathcal{N}  \\
    \phantom{\text{s.t.}} & \enskip \hat{f}_i, \hat{p}_i \geq 0, \enskip \forall i \in \mathcal{N}\\
    \phantom{\text{s.t.}} & \enskip \hat{z}_i \geq 0, \enskip \forall i \in \mathcal{N} \label{eq:placement_rudyn_posz}\\
    \phantom{\text{s.t.}} & \enskip \hat{z}_i = Wx_i , \enskip \forall i \in \mathcal{N} \label{eq:placement_rudyn_rulez}  
\end{align} \label{eq:placement_rudyn} 
\end{subequations}
It is apparent that the estimation problems \eqref{eq:placement_single_level} and \eqref{eq:placement_rudyn}, which BL and DR solve, respectively, are structurally different and so are $W^{\rm BL}$ and $W^{\rm DR}$ in general. For instance, think of a graph for which $\min\{g_\ell\}_{\ell \in \mathcal{A}}>\max\{h_b\}_{b \in \mathcal{B}}$. This represents a network where it is always cheaper to satisfy the nodal demand $y_b$, $b \in \mathcal{B}$, through the amount $z_b$ of product that is initially placed at the demand location, that is, a graph where product shipping would be uneconomical if the nodal demands were certainly known in advance. Indeed, take the $\ell$-$th$ row of $g+A^T\alpha_i$ in equation~\eqref{eq:placement_single_level_com2} for any $i \in \mathcal{N}$, that is, $g_{\ell}+\alpha_{o(\ell),i}-\alpha_{e(\ell), i}$, where $o(\ell)$ and $e(\ell)$ denote the origin and end nodes of arc $\ell$, respectively. We have that $\inf\{g_{\ell}+\alpha_{o(\ell),i}-\alpha_{e(\ell),i}: \alpha_{o(\ell),i} \in [0, h_{o(\ell)}],\alpha_{e(\ell),i} \in [0, h_{e(\ell)}] \} = g_{\ell}- h_{e(\ell)} >0$. Hence, $f_{\ell} = 0, \forall \ell \in \mathcal{A}$ and the system of inequalities \eqref{eq:placement_single_level_com1}-\eqref{eq:placement_single_level_com3} boils down to:
\begin{subequations}
\begin{align}
& \enskip 0 \leq (h-\alpha_i) \perp \hat{z}_i \geq 0, \enskip \forall i \in \mathcal{N} \label{eq:placement_com1}\\
& \enskip 0 \leq (\hat{z}_i-Wx_i) \perp \alpha_i \geq 0, \enskip \forall i \in \mathcal{N}  \label{eq:placement_com2}
\end{align}
\end{subequations}
which, unlike \eqref{eq:placement_rudyn_posz}--\eqref{eq:placement_rudyn_rulez}, allows for feasible solutions in the form $\hat{z}_{b,i} = 0$ with $w_{b}^T x_i < 0$ (and $\alpha_{i,b} = 0$), where $w_{b}$ is the $b$-$th$ row of matrix $W$. Furthermore, recasting \eqref{eq:placement_rudyn_rulez} as $\hat{z}_i - Wx_i = 0$ and setting $\alpha_i = h, \enskip \forall i \in \mathcal{N}$, it is trivial to see that any feasible point of DR is also feasible for BL. Since the feasible region of \eqref{eq:placement_rudyn} is contained in the feasible region of \eqref{eq:placement_single_level}, but the opposite is not true, the optimum of \eqref{eq:placement_single_level} is in general lower than  that of \eqref{eq:placement_rudyn}.

\subsection{Strategic producer problem}\label{sec:strategic}

Here we apply our decision-making framework to the problem of  a strategic producer partaking in a forward market \cite{Allaz1993}. This strategic player must decide the produced quantity $q\in \mathbb{R}$ that maximizes her profits while facing some uncertainty on market conditions. Let $c(q): \mathbb{R} \rightarrow \mathbb{R}^{+}$ denote the generation cost function whose parameters are assumed to be known with certainty. Let $p(q;Y): \mathbb{R} \times \mathbb{R}^m \rightarrow \mathbb{R}$ represent the inverse demand function expressing the impact of the generation quantity $q$ on the good's price. For some goods such as electricity, the inverse demand function varies depending on the season of the year, the day of the week, or the hour of the day. Besides, this function is also uncertain when producers must make their generation decisions $q$, since it may depend, for example, on weather conditions. If $Q$ represents the known feasible region of variable $q$ according to technical or economic constraints, the strategic producer must solve the following conditional stochastic optimization problem: 
\begin{align}
    \min_{q \in Q} & \enskip \mathbb{E} [c(q) - p(q;Y) q | X=x ] \label{eq:prod_general_obj}
\end{align}

As it is customary, we assume that the price and the demand are linearly related as $p(q; \alpha,\beta) = \alpha - \beta q$ where $\alpha \in \mathbb{R}$ and  $\beta \in \mathbb{R}^{+}$ are unknown parameters. Similarly, we assume that the production cost is computed through a quadratic cost function $c(q) = c_2 q^{2} + c_1 q$ where $c_1,c_2>0$ are known parameters related, respectively, to proportional production costs (such as fuel cost) and the increase of marginal costs due to technological factors (such as efficiency loss) \cite{Djurovic2012}. In order to ease the notation, we define $\alpha' = \alpha - c_1$ and  $\beta' = \beta + c_2$. Finally, we consider that the production quantity $q$ is bounded by known capacity limits, i.e., $\underline{q} \leq q \leq \overline{q}$ with $\underline{q}, \overline{q} \in \mathbb{R}^{+}$.  Thus, problem \eqref{eq:prod_general_obj} can be reformulated as:
\begin{align}
    \min_{\underline{q} \le q \le \overline{q}} & \enskip \mathbb{E} [\beta' q^{2} - \alpha' q | X=x]  \label{eq:m_prod1_obj}
\end{align}

Since the quantity decision $q$ is independent of the outcome of the uncertainty $(\beta', \alpha')$, the above can be further simplified to:
\begin{align}
    \min_{\underline{q} \le q \le \overline{q}} & \enskip \mathbb{E} [\beta'|X=x] q^{2} - \mathbb{E} [\alpha'| X=x] q   \label{eq:m_prod2_obj}
\end{align}

Therefore, the optimal solution $q^*$ is driven by the conditional expected values of $\alpha'$ and $\beta'$. To be more precise, since $\beta' >0$, $q^*$ could be equivalently computed as follows:
\begin{equation}
    q^*(x) \in \arg\min_{\underline{q} \le q \le \overline{q}}  \enskip  q^{2} - \frac{\mathbb{E} [\alpha'|x]}{\mathbb{E} [\beta'|x]}\, q  \implies q^*(x) \in \left\{\underline{q}, \frac{\mathbb{E} [\alpha'|x]}{2\mathbb{E} [\beta'|x]}, \overline{q}\right\} \label{eq:optimal_q_expectation}
\end{equation}

Unfortunately, $\mathbb{E} [\alpha'|x]$ and $\mathbb{E} [\beta'|x]$ are both unknown and therefore, they need to be estimated somehow. {\color{black} As explained further in Section \ref{subsec:exp_setup}}, the producer has available a set of historical observations  $S = \left\lbrace (\alpha'_{i}, \beta'_{i}, x_{i}),  \forall i \in \mathcal{N} \right \rbrace$ with $\alpha'_i \in \mathbb{R}$, $\beta'_{i} \in \mathbb{R}^+$ and $x_{i} \in \mathbb{R}^p$ {\color{black}in order to accomplish such a task}.
At this point, it should be underlined that the strategic producer problem~\eqref{eq:prod_general_obj} is of a distinctly different nature from that of the newsvendor problem~\eqref{eq:newsvendor} and the product placement problem~\eqref{eq:placement}. Indeed, the conditional stochastic program~\eqref{eq:prod_general_obj} \emph{has no recourse} and the uncertain parameters appear only in its objective function. Consequently, solving~\eqref{eq:prod_general_obj} is apparently as ``simple'' as estimating the two conditional expectations $\mathbb{E} [\alpha'|x]$ and $\mathbb{E} [\beta'|x]$. Our claim, however, is that the way the producer draws decisions from a \emph{finite data sample} (all we usually have in practice) may have a significant impact on the actual expected performance of the producer's strategy. Actually, the best estimates of $\mathbb{E} [\alpha'|x]$ and $\mathbb{E} [\beta'|x]$ from a \emph{statistical} sense do not necessarily result in the best offer $q$. 


According to the \textit{predict-then-optimize} strategies, the surrogate model of this problem is formulated as follows:
\begin{align}
    \min_{\underline{q} \le q \le \overline{q}} & \enskip \hat{\beta}' q^{2} - \hat{\alpha}' q   \label{eq:m_prod_surrogate}
\end{align}

%

As explained in Section \ref{sec:math_framework}, the traditional approach aims at learning the uncertain parameters $\alpha'_{i},\beta'_{i}$ as a function of the available information $x_{i}$. If we assume the family of linear functions, that is, $\hat{\alpha}'_{i} = w_{\alpha}^T x_{i}$, $\hat{\beta}'_{i} = w_{\beta}^T x_{i}$ with $w_{\alpha},w_{\beta}\in\mathbb{R}^p$, and we choose the squared error as the loss function $l^{\text {FO}}$, then the standard implementation of \eqref{eq:forecasting_step_theo} is: 
%
%
\begin{subequations}
\begin{align}
& w_{\alpha}^{\text {FO}} \in \arg \min_{w_{\alpha}\in\mathbb{R}^p} \sum_{i \in \mathcal{N}} (\alpha'_{i} - w_{\alpha}^T x_{i})^2 \label{eq:prod_forecasting_alpha}\\
& w_{\beta}^{\text {FO}} \in \arg \min_{w_{\beta}\in\mathbb{R}^p} \sum_{i \in \mathcal{N}} (\beta'_{i} - w_{\beta}^T x_{i})^2 \label{eq:prod_forecasting_beta}
\end{align} \label{eq:prod_forecasting}
\end{subequations}

The optimal quantity under context $X=x$ is the solution to the following optimization problem:
\begin{align}
     q^{\text {FO}}(x) \in \arg \min_{\underline{q} \le q \le \overline{q} } \enskip  (w_{\beta}^{\text {FO}})^T x q^{2} - (w_{\alpha}^{\text {FO}})^T x q \implies q^{\text {FO}}(x) \in \left\{\underline{q}, \frac{(w_{\alpha}^{\text {FO}})^T x}{2(w_{\beta}^{\text {FO}})^T x}, \overline{q}\right\} \label{eq:prod_forecasting_new_x}
\end{align}

Alternatively, $w_{\alpha}$ and $w_{\beta}$ can be determined following the proposed approach by solving the following bilevel formulation:
\begin{subequations}
\begin{align}
    w_{\alpha}^{\text {BL}}, w_{\beta}^{\text {BL}} \in \arg \min_{w_{\alpha},w_{\beta}\in\mathbb{R}^p} & \enskip \sum_{i \in \mathcal{N}} \beta'_{i} \hat{q}^{2}_{i} - \alpha'_{i} \hat{q}_{i} \label{eq:prod_bilevel_obj_up}\\
    \text{s.t.} & \enskip \hat{q}_{i} \in \arg \min_{\underline{q} \le q_{i} \le \overline{q}} w_{\beta}^T x_{i}q_{i}^{2} - w_{\alpha}^T x_{i} q_{i}, \enskip \forall i \in \mathcal{N} \label{eq:prod_bilevel_obj_dw}
\end{align} \label{eq:prod_bilevel}
\end{subequations}

For this particular application, the bilevel optimization problem rendered by the proposed approach has a significant drawback, because the global optimal solution of \eqref{eq:prod_bilevel} is $w_{\alpha}=w_{\beta}=0$. Consequently,  the lower-level problem \eqref{eq:prod_bilevel_obj_dw} can be replaced by the feasibility condition $\underline{q} \le \hat{q}_{i} \le \overline{q}$, and \textcolor{black}{the optimal values of $\hat{q}_{i}$ are determined as if uncertain parameters $\alpha'$ and $\beta'$ were known in advance, which violates non-anticipativity.} While this solution does lead to the minimum value of objective function \eqref{eq:prod_bilevel_obj_up}, it is useless to determine the optimal decisions for any context $X=x$. This degenerate solution of the proposed approach occurs because all coefficients of the objective function \eqref{eq:m_prod_surrogate} are uncertain. Interestingly, this shortcoming does not affect the newsvendor and product placement problems, because the uncertainty only affects the feasible region in those applications.

In this paper, we propose to ensure non-anticipativity  by formulating a bilevel optimization problem that considers the following modified surrogate model:
\begin{align}
    \min_{\underline{q} \le q \le \overline{q}} & \enskip q^{2} - \gamma q    \label{eq:m_prod_surrogate_gamma}
\end{align}

\noindent where $\gamma=\frac{\alpha'}{\beta'}$. For known values of $\alpha'$ and $\beta'$, the optimal solution of \eqref{eq:m_prod_surrogate} and \eqref{eq:m_prod_surrogate_gamma} coincide. However, surrogate model \eqref{eq:m_prod_surrogate_gamma} is simpler since it only includes one uncertain parameter instead of two. Assuming a linear relationship between the new uncertain parameter $\gamma$ and the contextual information, the proposed methodology yields the following bilevel problem:
\begin{subequations}
\begin{align}
    w_{\gamma}^{\text {BL}} \in \arg \min_{w_{\gamma}\in\mathbb{R}^p} & \enskip \sum_{i \in \mathcal{N}} \beta'_{i} \hat{q}^{2}_{i} - \alpha'_{i} \hat{q}_{i} \label{eq:prod_bilevel_gamma_up}\\
    \text{s.t.} & \enskip \hat{q}_{i} \in \arg \min_{\underline{q} \le q_{i} \le \overline{q}} q_{i}^{2} - w_{\gamma}^T x_{i} q_{i}, \enskip \forall i \in \mathcal{N} \label{eq:prod_bilevel_gamma_dw}
\end{align} \label{eq:prod_bilevel_gamma}
\end{subequations}

Formulation \eqref{eq:prod_bilevel_gamma} has the following advantages with respect to \eqref{eq:prod_bilevel}: i) it includes fewer parameters and therefore, it is less prone to overfitting, ii) it ensures non-anticipativity for any parameter vector $w_{\gamma}$, and iii) under certain conditions, it is able to retrieve the true model that relates random variable $\gamma$ and the context $X$ and the optimal solution to \eqref{eq:m_prod1_obj} as the sample size $|\mathcal{N}|$ grows to infinity, as shown in Proposition~\ref{prop:consistency} in~\ref{sec:appendix}. By replacing the lower-level problem with its KKT conditions, we obtain the following single-level problem:
\begin{subequations}
\begin{align}
    w_{\gamma}^{\text {BL}} \in \arg & \min_{w_{\gamma}\in\mathbb{R}^p}  \enskip \sum_{i \in \mathcal{N}} \beta'_{i} \hat{q}^{2}_{i} - \alpha'_{i} \hat{q}_{i} \\ 
    &\text{s.t.}  \enskip 2 \hat{q}_{i} - w_{\gamma}^T x_{i}  - \underline{\lambda}_i + \overline{\lambda}_i = 0, \enskip \forall i \in \mathcal{N}  \label{eq:prod_single_gamma_stationarity}\\
    &\phantom{s.t.}\enskip 0 \leq (\hat{q}_{i} - \underline{q}) \perp \underline{\lambda}_i \geq 0, \enskip \forall i \in \mathcal{N}  \label{eq:prod_single_gamma_com1}\\
    &\phantom{s.t.}\enskip 0 \leq (\overline{q} - \hat{q}_{i}) \perp \overline{\lambda}_i \geq 0, \enskip \forall i \in \mathcal{N}  \label{eq:prod_single_gamma_com2}
\end{align} \label{eq:prod_single_gamma}
\end{subequations}
\noindent where $\underline{\lambda}_i,\overline{\lambda}_i$ are the dual variables corresponding to the capacity limit constraints. Notice that if complementarity conditions \eqref{eq:prod_single_gamma_com1} and \eqref{eq:prod_single_gamma_com2} are reformulated using the Fortuny-Amat approach, problem  \eqref{eq:prod_single_gamma} can be solved to global optimality as a quadratic mixed-integer program using off-the-shelf optimization software. According to this procedure, optimal decisions under context $X=x$ are made by solving:

\begin{align}
    q^{\text {BL}}(x) \in \arg \min_{\underline{q} \le q \le \overline{q} } & \enskip  q^{2} -  (w_{\gamma}^{\text {BL}})^T x q \implies q^{\text {BL}}(x) \in \left\{\underline{q}, \frac{(w_{\gamma}^{\text {BL}})^T x}{2}, \overline{q}\right\} \label{eq:prod_bilevel_new_x}
\end{align}

Finally, we can directly learn the optimal production as a function of the known information as proposed in \cite{Ban2019}. Assuming the linear mapping $\hat{q}_{i} = w_q^T x_{i}$ with $w_q\in\mathbb{R}^p$, problem \eqref{eq:decision_rule_framework} for this particular application is formulated as:
\begin{subequations}
\begin{align}
    w_q^{\text {DR}} \in \arg \min_{w_q \in \mathbb{R}^p} & \sum_{i \in \mathcal{N}} \beta'_{i} (w_q^T x_{i})^2 - \alpha'_{i} w_q^T x_{i} \label{eq:prod_cynthia_obj} \\
    \text{s.t.} & \enskip \underline{q} \leq w_q^T x_{i} \leq \overline{q} \enskip \forall i \in \mathcal{N} \label{eq:prod_cynthia_capacity}
\end{align} \label{eq:prod_cynthia} 
\end{subequations}

Formulation \eqref{eq:prod_cynthia} is a convex quadratic optimization problem and can be then solved using commercial software such as CPLEX. In line with \eqref{eq:deci_rule}, the optimal quantity under context $X=x$ is directly computed as:

\begin{align}
    q^{\text {DR}}(x) = (w_q^{\text {DR}})^T x \label{eq:prod_decision_rule_new_x}
\end{align}

Although not true in general, approaches \eqref{eq:prod_single_gamma} and \eqref{eq:prod_cynthia} may lead to the same solution under specific conditions. For instance, if the produced quantity $q$ is not limited by minimum/maximum bounds, then constraint \eqref{eq:prod_single_gamma_stationarity} boils down to $\hat{q}_i = w_{\gamma}^T x_i / 2$. Consequently, the solutions of \eqref{eq:prod_single_gamma} and \eqref{eq:prod_cynthia} satisfy that $w_q^{\text {DR}}=w_{\gamma}^{\text {BL}}/2$ and therefore, $q^{\text {BL}}(x) = q^{\text {DR}}(x)$ for any context $X=x$. As we show in the following section, the decisions $q^{\text {BL}}$ delivered by our approach are significantly more profitable than $q^{\text {DR}}$ in the \emph{constrained} case.

\section{Numerical simulations} \label{sec:simulations}

As an additional contribution, we assess and compare the performance of the proposed approach for the strategic producer problem using numerical simulations. In Section \ref{subsec:examples} we illustrate the advantages of BL with respect to FO and DR using a small example with a reduced data sample. Additionally, Section \ref{sec:case_study} presents the numerical results of a realistic case study that uses real data from the Iberian Electricity Market and the Spanish Transmission System Operator \cite{Omie, Entsoe}.

\subsection{Illustrative example}
\label{subsec:examples}
This section aims at gaining insight into the performance of the proposed approach  with a small example of the strategic producer problem. For the sake of simplicity, we only consider four realizations of the uncertain parameters $\alpha'_{i},\beta'_{i}$ and a single feature $x_i\in[0,10]$, whose values are shown in Table~\ref{table:toy1_data}. Approach FO predicts the uncertain parameters using linear functions in the form $\hat{\alpha}_i=w_{\alpha,0}^{\rm FO} + w_{\alpha,1}^{\rm FO}x_i$ and $\hat{\beta}_i=w_{\beta,0}^{\rm FO} + w_{\beta,1}^{\rm FO}x_i$; approach BL assumes that $\hat{\gamma}_i = w_{\gamma,0}^{\rm BL} + w_{\gamma,1}^{\rm BL}x_i$; and approach DR considers $\hat{q}_i = w_{q,0}^{\rm DR} + w_{q,1}^{\rm DR}x_i$. These three approaches are compared with \textcolor{black}{a benchmark method (BN) that assumes perfect knowledge of the uncertain parameters $\alpha',\beta'$} and, consequently, yields the best possible offer for each time period. Obviously, this method cannot be implemented in practice and, accordingly, is just used here for comparison purposes. Given the reduced size of this example, methods BL-R and BL-M provide the same results and are thus jointly referred to as BL.

\begin{table}

\setlength{\tabcolsep}{12pt}
\centering
\begin{tabular}{ccccc}
\hline
$i$ & $x_{i}$ &  $\alpha'_{i}$ & $\beta'_{i}$ & $\gamma_{i}$ \\
\hline
1 &  2  &	 2  &  10   & 0.20 \\
2 &  4  &	17  &  10   & 1.70 \\
3 &  8  &	 8  &   3   & 2.67 \\
4 &  9  &   16  &   6   & 2.67 \\
\hline
\end{tabular}
\captionsetup{justification=centering}
\caption{Data sample $S$ for the illustrative example.}
\label{table:toy1_data}
\end{table}

First, we deal with the \emph{unconstrained case}, that is, the case in which the capacity constraints are disregarded. Table~\ref{table:toy1_incomes} shows the in-sample results obtained from methods BN, FO, DR, and BL, namely, the optimal production quantity for each time period $q_{i}$, the absolute income (I), and the relative income with respect to the benchmark (RI).  Notice that the income for each time period can be computed as $- \beta'_{i} q_{i}^{2} + \alpha'_{i} q_{i}$. As discussed in Subsection~\ref{sec:strategic}, in connection with the unconstrained case, coefficients $w^{\text {DR}}$ are equal to $w^{\text {BL}} / 2$ and the decisions and incomes obtained by DR and BL are the same as a result.  It is also interesting that the income of these two methods is 5\% higher than that of FO. To explain this, we refer to Fig.~\ref{fig:toy_example_un}, which depicts the optimal production quantities given by the different methods as a function of the context $x \in [0,10]$, namely, 
\begin{equation}
q^{\text {FO}}(x) = \frac{w_{\alpha,0}^{\rm FO} + w_{\alpha,1}^{\rm FO}x}{2(w_{\beta,0}^{\rm FO} + w_{\beta,1}^{\rm FO}x)} \quad q^{\text {BL}}(x) = \frac{w_{\gamma,0}^{\rm BL} + w_{\gamma,1}^{\rm BL}x}{2} \quad q^{\text {DR}}(x) = w_{q,0}^{\rm DR} + w_{q,1}^{\rm DR}x
\end{equation}

This figure shows that methods BL and DR can return decisions much closer to the benchmark ones than method FO for the four data points in the sample. Therefore, even for unconstrained optimization problems, the proposed methodology may outperform the classical ``first-predict-then-optimize'' approach, which is purely based on reducing the error of forecasting the uncertain parameters, simply because minimizing this error is not necessarily aligned with maximizing the decision value. 

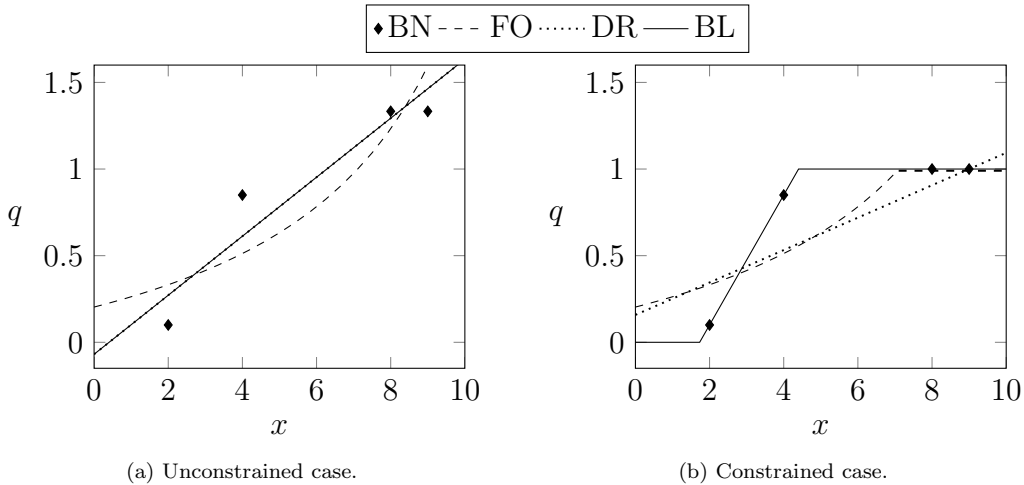
\begin{figure} 
\centering
\captionsetup[subfigure]{justification=centering}
    \begin{subfigure}[b]{0.475\textwidth}
        \centering
        \begin{tikzpicture}[]
        	\begin{axis}[
        	width=\textwidth,
            xmin = 0, xmax = 10, ymin=-0.15, ymax=1.6,
        	legend style={at={(1.25,1.04)},anchor=south,legend cell align=center,legend columns=4},	
        	clip marker paths=true,	
        	xlabel = $x$,
        	ylabel = $q$,
        	ylabel style={rotate=-90}]
        	\addplot[only marks, mark=diamond*] coordinates {(2,0.1) (4,0.85) (8,1.333) (9,1.333)}; \addlegendentry{BN}
        	\addplot[mark=none, dashed, domain=0:10] {0.5 * (5.000 + 1.000 * x ) / (12.29770992 + -0.877862595 * x)}; \addlegendentry{FO}
        	\addplot[mark=none, dotted, domain=0:10, line width=0.75pt] {-0.069024618 + 0.170302275 * x}; \addlegendentry{DR}
        	\addplot[mark=none, solid, domain=0:10] {0.5 * (-0.138049237 + 0.340604550 * x)}; \addlegendentry{BL}
        	\end{axis}	
        \end{tikzpicture} \\
        \caption{Unconstrained case.}
        \label{fig:toy_example_un}
    \end{subfigure}
    \hfill
    \begin{subfigure}[b]{0.475\textwidth}
        \centering
        \begin{tikzpicture}[]
        	\begin{axis}[	
        	width=\textwidth,
            xmin = 0, xmax = 10, ymin=-0.15, ymax=1.6,
        	clip marker paths=true,	
        	xlabel = $x$,
        	ylabel = $q$,
        	ylabel style={rotate=-90}]
        	\addplot[only marks, mark=diamond*] coordinates {(2,0.1) (4,0.85) (8,1) (9,1)}; 
        	\addplot[mark=none, dashed, domain=0:7.1] {0.5 * (5.000 + 1.000 * x ) / (12.29770992 + -0.877862595 * x)}; 
        	\addplot[mark=none, dotted, line width=0.75pt, domain=0:10] {0.158142665311473 + 0.0935397037399315 * x}; 
        	\addplot[mark=none, solid] coordinates {(0,0) (1.733,0) (4.4, 1) (10,1)}; 
        	\addplot[mark=none, dashed, line width=0.75pt] coordinates {(7.1,0.99) (10,0.99)};
        	\end{axis}	
        \end{tikzpicture} \\
        \caption{Constrained case.}
        \label{fig:toy_example_con}
    \end{subfigure}
\caption{Decision quantity $q$ versus feature $x$ for the illustrative example.}
\label{fig:toy_example}
\end{figure}

\begin{table}[]
    \setlength{\tabcolsep}{6pt}
    \centering
    \begin{tabular}{ccccccc}
    \hline
     & $q_1$ & $q_2$ & $q_3$ & $q_4$ & I(\euro) & RI(\%)\\
    \hline
    BN    & 0.10 & 0.85 & 1.33 & 1.33 & 23.33 & 100.0 \\
    FO    & 0.33 & 0.51 & 1.23 & 1.59 & 21.21 & 91.0  \\
    DR    & 0.27 & 0.61 & 1.29 & 1.46 & 22.36 & 95.9  \\
    BL    & 0.27 & 0.61 & 1.29 & 1.46 & 22.36 & 95.9  \\
    \hline 
    \end{tabular}
    \caption{Optimal offer and income for the unconstrained illustrative example (in-sample results).  Parameter vectors $w$ are: $w_{\alpha, 0}^{\text{FO}} = 5.000, w_{\alpha, 1}^{\text{FO}} = 1.000, w_{\beta, 0}^{\text{FO}} = 12.298, w_{\beta, 1}^{\text{FO}} = -0.878, w_{\gamma,0}^{\text{BL}} = -0.138$, $w_{\gamma,1}^{\text{BL}} = 0.341, w_{q,0}^{\text{DR}} = -0.069$, $w_{q,1}^{\text{DR}} = 0.170$.}
    \label{table:toy1_incomes}
\end{table}

Now we consider the \emph{constrained case}, that is, we bring the capacity constraints back into this small example. In particular, the minimum and maximum outputs of the strategic producer are set to 0 and 1, respectively. Similarly to Table~\ref{table:toy1_incomes}, the in-sample results obtained in the capacity-constrained case are collated in Table~\ref{table:toy2_incomes}, where we can see that the optimal quantity $q_{i}$ reaches its maximum value for some time periods and methods FO, DR and BL all provide different results. Methods FO and DR achieve an income 7.5\% and 8.2\% lower than the benchmark. This poor in-sample performance is better understood by means of Fig.~\ref{fig:toy_example_con}, which similarly to Fig.~\ref{fig:toy_example_un}, represents the optimal quantities as a function of the context for the constrained case according to \eqref{eq:prod_forecasting_new_x}, \eqref{eq:prod_bilevel_new_x} and \eqref{eq:prod_decision_rule_new_x}. First, since method FO is unaware of the feasibility region of the original conditional stochastic problem, it provides the same prediction of the uncertain parameters $\alpha,\beta$ in the unconstrained and constrained cases. However, using these forecasts in the surrogate model \eqref{eq:m_prod_surrogate} enforces $q=1$ for $x \ge 7.1$ in the constrained case. As observed, reducing the forecast error of $\alpha,\beta$ does not lead to the maximization of the decision value in the constrained case either. Second, method DR must ensure feasible solutions for all samples, a condition that also leads to quite poor approximations of the optimal quantities for most values of the context $x$. Furthermore, this approach would return infeasible solutions $q>1$ for $x>9$ as shown in Fig.~\ref{fig:toy_example_con}. On the contrary, the proposed approach BL can find a linear relation between $\gamma$ and $x$ to be used in the surrogate model \eqref{eq:m_prod_surrogate_gamma} that results in decisions $q$ that perfectly match those provided by the benchmark for the four data points and therefore, this method achieves the highest possible income in sample.

\begin{table}[]
    \setlength{\tabcolsep}{6pt}
    \centering
    \begin{tabular}{ccccccc}
    \hline
     & $q_1$ & $q_2$ & $q_3$ & $q_4$ & I(\euro) & RI(\%)\\
    \hline
    BN    & 0.10 & 0.85 & 1.00 & 1.00 & 22.33 & 100.0 \\
    FO    & 0.33 & 0.51 & 1.00 & 1.00 & 20.65 &  92.5 \\
    DR    & 0.35 & 0.53 & 0.91 & 1.00 & 20.50 &  91.8 \\
    BL  & 0.10 & 0.85 & 1.00 & 1.00 & 22.33 & 100.0 \\
    \hline 
    \end{tabular}
    \caption{Optimal offer and income for the constrained illustrative example (in-sample results). Parameter vectors $w$ are: $w_{\alpha, 0}^{\text{FO}} = 5.000, w_{\alpha, 1}^{\text{FO}} = 1.000, w_{\beta, 0}^{\text{FO}} = 12.298, w_{\beta, 1}^{\text{FO}} = -0.878, w_{\gamma,0}^{\text{BL}} = -1.300$, $w_{\gamma,1}^{\text{BL}} = 0.750, w_{q,0}^{\text{DR}} = 0.158$, $w_{q,1}^{\text{DR}} = 0.094$.}
    \label{table:toy2_incomes}
\end{table}

In summary, this small example  sheds light on the reasons why the proposed methodology outperforms existing ones for both unconstrained and constrained optimization problems under uncertainty: Our approach provides forecasts of the uncertain parameters that take into account the objective function and feasible region of the decision maker. Such enhanced forecasts translate into decisions that are much closer to those obtained in the ideal perfect information instance.

\subsection{Case study}
\label{sec:case_study}

In this section, we compare the proposed approach with existing ones using realistic data from the Iberian electricity market, as described in detail in  Section \ref{subsec:exp_setup}. Sections \ref{subsec:impact_technology}, \ref{subsec:impact_c2} and \ref{subsec:impact_flexibility} investigate how the type of generation portfolio, the quadratic cost term $c_2$, and the residual demand elasticity impact the performance of the proposed methodology, respectively. These three subsections only include the global optimal solutions given by method BL-M. Finally, Section~\ref{subsec:computational} provides computational solution times for all the approaches and discusses the differences between BL-R and BL-M in that respect.

\subsubsection{Experimental setup}
\label{subsec:exp_setup}


In order to test our proposal, we consider a realistic case study based on actual data from the Iberian electricity market. We construct a data set of the form $\{(x_i, \alpha_i, \beta_i), \forall i \in \mathcal{N}\}$ from which we derive the rest of the parameters required for our simulations as explained in Section \ref{sec:strategic}. We gather raw market data from the forward market OMIE \cite{Omie} to compute parameters $\alpha_i$, $\beta_i$ of the inverse demand function. Furthermore, we collect wind and solar power forecasts of the aggregated production of Spain to be used as a vector of contextual information $x_i$. The wind and solar forecasts, originally published by the Spanish TSO, are downloaded from the ENTSO-e Transparency Platform \cite{Entsoe}.

Historical raw hourly block-wise bids and offers submitted by buyers and sellers to the Iberian day-ahead energy market are processed to obtain parameters $\alpha_i$, $\beta_i$ as follows. For each hour of the year, we have access to the set of bids and offers defined as $\{(q_b, p_b), \forall b \in B\}$, $\{(q_o, p_o), \forall o \in O\}$, respectively, where $q_{b/o}$ is the amount of energy to buy/sell at price $p_{b/o}$. Thus, the residual demand $r$ to be potentially covered by a new producer entering the market for each possible price $p$ is defined as $r:=\sum_{b \in B: p_b \ge p } q_b - \sum_{o \in O: p_o \le p} q_o$, that is, the aggregated demand minus the aggregated production. The step-wise function relating the residual demand $r$ and the electricity price $p$ is plotted in Fig.~\ref{fig:market_curve} for illustrative purposes. 


Now consider that a new strategic producer enters the market with an offer to sell quantity $q$ at offer price $0$. If we assume that the remaining bids and offers stay unaltered, the market price would decrease following the right-hand part of the step-wise function depicted in Fig.~\ref{fig:market_curve}. Therefore, a strategic producer aiming at maximizing her profit is interested in modeling the dependence between her offered quantity $q$ and the market price $p$ in the shaded area, with parameter $\delta$ being a constant sufficiently larger than the producer's maximum generation capacity. In connection with Section \ref{sec:strategic}, we approximate said dependency using a linear function such that $p_i(q)=\alpha_i - \beta_i q$ as illustrated in Fig.~\ref{fig:market_curve} and therefore, the values of $\alpha_i,\beta_i$ for each hour are obtained by determining the linear function that best approximates the blocks shaded in gray.


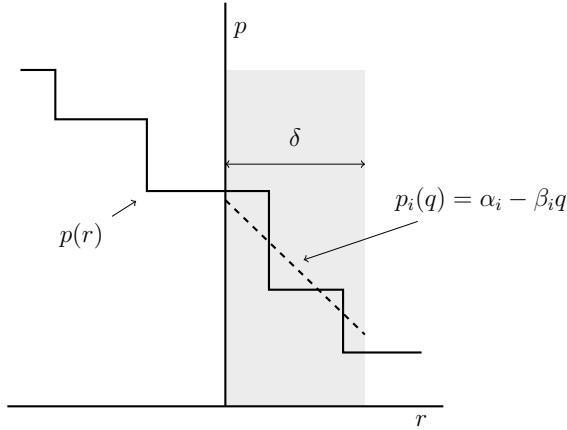
\begin{figure} 
\def\dh{2.7}
\def\dw{3.2}
\centering
    \begin{tikzpicture}[scale=0.8]
    	\begin{axis}[
    	    scale=1,
        	hide axis,
        	width=11cm,
            xmin = -5, xmax = 8, ymin=-0.3, ymax=5,
        	clip marker paths=true]
            \addplot+[mark=none, draw=none, fill=lightgray!30] coordinates {(0, 3.75) (\dw, 3.75)} \closedcycle;
            
        	\addplot[mark=none, solid, line width=1pt] coordinates {(-5, 0) (5,0)};
        	\addplot[mark=none, solid, line width=1pt] coordinates {(0, 0) (0,4.5)};
        	\node[anchor=west] (pp) at (axis cs:0,4.2){$p$};
            \node[anchor=north] (qq) at (axis cs:4.5,0){$r$};
            
        	\addplot[mark=none, solid, domain=-5:5, line width=1pt] coordinates {
        	(-4.7,3.75) (-3.9,3.75) (-3.9,3.2) (-1.8,3.2) (-1.8,2.4) (1,2.4)
        	(1, 1.3) (2.7, 1.3) (2.7, 0.6) (4.5, 0.6)};
        	
            \addplot[mark=none, dashed, line width=1pt] coordinates {(0, 2.3) (\dw,0.8)};
            
            \addplot[mark=none, solid, <->] coordinates {(0, \dh) (\dw, \dh)};

            \node[anchor=south] (delta) at (axis cs:\dw/2,\dh+0.1){$\delta$};
            
            \node[anchor=west] (aa) at (axis cs:3.7,2.3){$p_i(q) = \alpha_i - \beta_i q$};
            \node (bb) at (axis cs:1.6,1.6){};
            \draw[->](aa)--(bb);
            
            \node[anchor=west] (cc) at (axis cs:-4.0,1.9){$p(r)$};
            \node (dd) at (axis cs:-1.85,2.35){};
            \draw[->](cc)--(dd);
            
    	\end{axis}	
    \end{tikzpicture} \\
\caption{Inverse residual demand curve $p(r)$ (solid) and fitted inverse demand function $p_i(q)$ (dashed) in the interval $[0, \delta]$ . The intercept and slope of the fitted line are $\alpha_i$ and $-\beta_i$, respectively.}
\label{fig:market_curve}
\end{figure}



We collect data from November 2018 to October 2019 in order to build a data set of 8600 hours (almost one year), which is divided into 43 bins of 200 consecutive samples. \textcolor{black}{Each bin is randomly split into training and test sets with a ratio of $80\% / 20 \%$, respectively. This process is repeated five times for each bin. Therefore, each approach is solved for 215 different training sets of 160 samples, and the obtained solutions are evaluated using the corresponding 215 test sets of 40 samples each. The out-of-sample results provided in Sections \ref{subsec:impact_technology}, \ref{subsec:impact_c2}, \ref{subsec:impact_flexibility}, and \ref{subsec:computational} are obtained by averaging the outcomes over these 215 test sets.} We select a value of $\delta$ equal to 5 GW in order to encompass enough bids and offers to obtain accurate approximations of $p_i(q)$ throughout the whole data set. We determine the optimal parameters $w$ through problems \eqref{eq:prod_forecasting}, \eqref{eq:prod_single_gamma}, and \eqref{eq:prod_cynthia}, which we denote FO, BL and DR, respectively. More specifically, we name BL-M the Fortuny-Amat big-M reformulation of model \eqref{eq:prod_single_gamma} and BL-R the regularized counterpart, which are a particularization of \eqref{eq:kkt_reformulation} and \eqref{eq:kkt_reformulation_reg}, respectively.

Each bin is executed in parallel with the following resources: 1 CPUs Intel E5-2670 @ 2.6 GHz and 1 Gb of RAM. 
Each instance of model BL-M is solved using the MIQP solver CPLEX \cite{CPLEX} for a maximum time of 20 minutes or a relative gap of $10^{-8}$. On the other hand, BL-R is executed using the NLP solver CONOPT \cite{CONOPT} without time limit. 

\subsubsection{Impact of the generation portfolio}\label{subsec:impact_technology}

As previously stated, the main advantage of our approach is that it yields forecast values for the uncertain parameters that are tailored to the optimization problem by which the strategic power producer determines her optimal market sale. However, such an advantage may translate into higher or lower incomes  depending on the firm's generation portfolio. In this section, therefore, we evaluate the performance of the various approaches for three generic power plants characterized by different linear costs ($c_1$) and capacities ($\overline{q}$).

Table \ref{tab:units} provides the values of $c_1$, $c_2$ and $\overline{q}$ for these three  generic units. For simplicity, the minimum output $\underline{q}$ of all units is assumed equal to 0 and the value of $c_2$ is set to 0.005 \euro/MWh$^2$ \citep{Djurovic2012}. The base unit can represent a nuclear power station and is characterized by low fuel cost and high capacity. The medium unit can be, for example, a carbon-based power station with a lower capacity and higher fuel costs. Finally, peak units, such as combined cycle power plants, typically have the highest fuel cost and a smaller generation capacity. Table \ref{tab:units} also includes the percentage of time periods in which $q^{\text{BN}}=0$, $0 < q^{\text{BN}} < \overline{q}$ and $q^{\text{BN}}=\overline{q}$ denoted as $\mathcal{N}_{q^{\text{BN}}=0}$, $\mathcal{N}_{0 < q^{\text{BN}} < \overline{q}}$, and $\mathcal{N}_{q^{\text{BN}}=\overline{q}}$, respectively, where $q^{\text{BN}}$ represents the optimal quantity that the strategic firm would place into the market under the \emph{true} inverse demand function (that is, the solution given by the benchmark approach). It is observed that the base unit generates at maximum capacity for most times periods and  is only shut down in 8\% of the cases. The medium generating unit is idle 32\% of the time (if prices are too low) and is at maximum capacity during the 39\% of the time periods. Finally, the peak unit is not dispatched most of the time since electricity prices are usually below its marginal production cost.

\begin{table}[]
    \centering
    \begin{tabular}{lcccccc}
    \hline
    & $c_1$ & $c_2$ & $\overline{q}$ & $\mathcal{N}_{q^{\text{BN}}=0}$ & $\mathcal{N}_{0 < q^{\text{BN}} < \overline{q}}$ & $\mathcal{N}_{q^{\text{BN}} = \overline{q}}$ \\
    & (\euro/MWh) & (\euro/MWh$^2$) & (MW) & (\%) & (\%) & (\%) \\
    \hline
    Base   & 10 & 0.005 & 1000 &  8  &  16  & 76  \\ 
    Medium & 35 & 0.005 & 500  & 32  &  29  & 39  \\
    Peak   & 50 & 0.005 & 250  & 79  &  12  &  9  \\
    \hline
    \end{tabular}
    \caption{Generation technology data.}
    \label{tab:units}
\end{table}

\textcolor{black}{Table \ref{tab:results_base_medium_peak} provides the out-of-sample results computed by averaging over the 215 test sets of 40 samples each described in Section \ref{subsec:exp_setup}.} These results include the absolute income for the benchmark approach (I$^{\text {BN}}$) for the considered time horizon, the relative income (RI) for methods FO, DR and BL-M, and the percentage of time periods for which method DR provides infeasible solutions (INFES$^{\text {DR}}$). A first obvious observation is that, as expected, the absolute income is higher for base units and lower for peak units. A second, probably more interesting remark relates to the impact of the uncertainty about the inverse demand function on the market revenues accrued by each generating technology. Since the base unit is at full capacity most of the time, the uncertainty pertaining to the residual demand does not affect revenues that much, and the three methods obtain relative incomes above 94\%. On the contrary, the participation of the medium and peak units highly depends on market conditions and therefore, this very same uncertainty  remarkably deteriorates market revenues, with the eventual result that the maximum relative incomes amount to 80\% and 59\%, respectively, for the method featuring the best performance (which is BL-M). 

\begin{table}[]
    \centering
\begin{tabular}{lccccc}
    \hline
       & I$^{\text {BN}}$  & RI$^{\text {FO}}$ &RI$^{\text {DR}}$ & RI$^{\text {BL-M}}$ & INFES$^{\text {DR}}$ \\
       & (M\euro) & (\%) & (\%) & (\%) & (\%) \\
    \hline
    Base   & 176.7 & 96.1 & 94.6 & 96.3 & 4.9 \\
    Medium  & 20.9 & 77.4 & 62.5 & 80.0 & 1.7 \\
    Peak    &  1.2 & 44.1 & 18.9 & 58.7 & 0.1 \\
    \hline
\end{tabular}
    \caption{Impact of generation technology (Out-of-sample} results).
    \label{tab:results_base_medium_peak}
\end{table}


On a different front, the DR approach produces infeasible offers in a considerable number of time periods, whereas FO and BL-M are guaranteed to provide feasible production quantities in all cases. The percentage of periods for which method DR results in an infeasible $q$ is higher for the base unit because the medium and peak units are idle more frequently. For this particular application, making DR decisions feasible can be easily achieved by computing $\min(\max(\hat{q}_{i},\underline{q}),\overline{q})$. However, this post-processing step to guarantee feasibility can be much more challenging in applications with general convex feasible sets. It is also apparent that the DR approach provides the lowest RI for the three cases considered and therefore, this method is not even recommended for decision-making models where the decision vector is simply bounded component-wise.

Finally, we notice that, for the three generation technologies, the proposed method BL-M always provides higher incomes than the FO approach. However, relative income  improvements vary widely for each case. For the base unit, the relative income of BL-M is only 0.2\% higher than that of FO. This is understandable since this power plant is at full capacity most of the time and thus, the impact of the uncertainty is comparatively minor, as we mentioned before. For the peak unit, in contrast, the relative income of BL-M is 14.6\% higher than that of FO. Note that, unlike for base units, making small errors in the forecasts of the market conditions can be catastrophic for peak units, because such deviations may mean the difference between producing nothing or producing at maximum capacity. The ability of BL-M to reduce the forecast error when consequences are worse, together with the lower absolute incomes of peak units, explains this high difference in percentage. The gain of BL-M with respect to FO for the medium unit has an intermediate value of 2.6\%.

To conclude this section, Table \ref{tab:incomes_losses} includes, for the peak generating unit, the percentage of periods with a positive income, with a negative income and with an income equal to zero, denoted as $\mathcal{N}_{\text{I}>0}$, $\mathcal{N}_{\text{I}<0}$ and $\mathcal{N}_{\text{I}=0}$, in that order. The total sum of positive and negative incomes is also provided in the last two rows, represented by the symbols I$^+$ and I$^-$, respectively. Interestingly, BL-M achieves the highest percentage of periods with a positive income and succeeds in providing the highest value of I$^+$.

\begin{table}[]
    \centering
    \begin{tabular}{ccccc}
    \hline
    & BN & FO & DR & BL-M \\
    \hline
    $\mathcal{N}_{\text{I}>0}(\%)$   & 20.6 &  9.0 &  6.6 & 10.1 \\
    $\mathcal{N}_{\text{I}<0}(\%)$   &  0.0 &  3.7 &  3.0 &  3.4 \\
    $\mathcal{N}_{\text{I}=0}(\%)$   & 79.4 & 87.3 & 90.4 & 86.4 \\
    I$^+$(M\euro)        &  1.23 &  0.73 &  0.37 &  0.87 \\
    I$^-$(M\euro)        &  0.00 & -0.19 & -0.14 & -0.15 \\
    \hline
    \end{tabular}
    \caption{Income distribution for the peak generating unit (Out-of-sample} results).
    \label{tab:incomes_losses}
\end{table}

\subsubsection{Impact of parameter \texorpdfstring{$c_2$}{c2}} \label{subsec:impact_c2}

While parameter $c_1$ basically depends on the cost of the fuel used by each unit, the interpretation of $c_2$ is not as straightforward. Indeed, this parameter measures the decrease in the plant marginal cost as production increases and is connected to technological aspects of the plant's economy of scale, like the way the plant efficiency varies for different operating points. For this reason, in this section, we investigate the impact of $c_2$ on the performance of the proposed method. Notice that, if $\underline{q}=0$ MW, then the unit marginal costs range from $c_1$ to $c_1+c_2\overline{q}$. In a similar way, as Table \ref{tab:units} does, Table \ref{tab:below_above_c2} shows the operating regime of a medium generating unit with $c_1=35$\euro/MWh, $\overline{q}=500$ MW and different values of $c_2$. As expected, a decrease in $c_2$ entails a reduction in the marginal production cost of the plant and, as a result, the amount of electricity the strategic firm places into the market increases.

\begin{table}[]
    \centering
\begin{tabular}{cccc}
    \hline
    $c_2$   & $\mathcal{N}_{q^{\text{BN}}=0}$ & $\mathcal{N}_{0 < q^{\text{BN}} < \overline{q}}$ & $\mathcal{N}_{q^{\text{BN}} = \overline{q}}$ \\
    (\euro/MWh$^2$) & (\%) & (\%) & (\%) \\
    \hline
    0.01   & 32 & 43 & 25 \\
    0.005  & 32 & 29 & 39 \\
    0.001  & 32 & 15 & 53 \\
    \hline
\end{tabular}
    \caption{Operating regime of a medium generating unit ($c_1=35$\euro/MWh, $\overline{q}=500$MW). }
    \label{tab:below_above_c2}
\end{table}

Table \ref{tab:results_medium} provides the same results as Table \ref{tab:results_base_medium_peak}, but for different values of $c_2$ and the medium generating unit only. Naturally, reducing the plant marginal costs increases both the absolute income for the benchmark approach and also the relative income achieved by all methods. 
Nevertheless, BL-M proves to be between 1.8\% and 2.6\% more profitable to the producer than the traditional FO approach for the values of $c_2$ considered.

\begin{table}[]
    \centering
\begin{tabular}{cccccc}
    \hline
    $c_2$   & I$^{\text {BN}}$ & RI$^{\text {FO}}$ & RI$^{\text {DR}}$ & RI$^{\text {BL-M}}$ & INFES$^{\text {DR}}$ \\
    (\euro/MWh$^2$) & (M\euro)  & (\%) & (\%) & (\%) & (\%)  \\
    \hline
    0.01   & 16.3  & 73.9 & 60.0 & 76.4 & 1.1 \\
    0.005  & 20.9  & 77.4 & 62.5 & 80.0 & 1.7 \\
    0.001  & 25.5  & 81.2 & 65.6 & 83.0 & 1.0 \\
    \hline
\end{tabular}
    \caption{Impact of parameter $c_2$ on a medium generating unit (Out-of-sample} results).
    \label{tab:results_medium}
\end{table}


\subsubsection{Impact of the residual demand elasticity} \label{subsec:impact_flexibility}

So far we have centered our study on the cost structure of the generation portfolio owned by the strategic firm. Here, on the contrary, we focus on the elasticity of the market residual demand. Roughly speaking, this elasticity is inversely proportional to parameter $\beta$ of the inverse demand function. Bearing this in mind, we compare the next two market situations, namely, the ``Normal'' and the ``Low-elast'' instances. The former corresponds to the values of $\beta$ in the original data set, while the latter is obtained by multiplying these $\beta$-values by two. 

Table~\ref{tab:2xbeta} shows the incomes provided by each of the considered methods relative to those of the benchmark. The numbers correspond to the medium power plant of Table~\ref{tab:units}. The overall effect of increasing the residual demand elasticity (lower $\beta$-values) is analogous to that of decreasing parameter $c_2$, i.e., the involvement of the strategic producer in the market augments, thus leading to higher revenues. Results in Table~\ref{tab:2xbeta} show that the proposed BL approach outperforms FO and DR for different values of the residual demand elasticity, improving the competitive edge of the strategic producer in more than 2\% with respect to FO in terms of relative income.


\begin{table}[]
    \centering
\begin{tabular}{lccccc}
    \hline
       & I$^{\text{BN}}$ & RI$^{\text{FO}}$ & RI$^{\text{DR}}$ & RI$^{\text{BL-M}}$ & INFES$^{\text{DR}}$ \\
       & (M\euro) & (\%)  & (\%) & (\%) & (\%)  \\
    \hline
    Normal    & 20.9  & 77.4 & 62.5 & 80.0 & 1.7 \\
    Low-elast & 18.6  & 74.7 & 60.0 & 77.1 & 1.7 \\
    \hline
\end{tabular}
    \caption{Impact of residual demand elasticity  (Out-of-sample} results).
    \label{tab:2xbeta}
\end{table}


\subsubsection{Computational results}
\label{subsec:computational}

In Sections \ref{subsec:impact_technology}, \ref{subsec:impact_c2} and \ref{subsec:impact_flexibility} we have only included results from BL-M, and not from BL-R, because the former variant of the bilevel framework we propose guarantees global optimality for the strategic producer problem for appropiate values of large constants $M^P, M^D$. However, solving model BL-M can be computationally very expensive. Alternatively, local optimal solutions of the proposed bilevel model \eqref{eq:prod_single_gamma} can be efficiently found by way of the particularization of the regularization approach \eqref{eq:kkt_reformulation_reg} that we named BL-R. 

Next, we first compare the solutions given by methods BL-M and BL-R. In order to solve model BL-R, we iteratively shrink the regularization parameter $\epsilon$ taking values from the discrete set $\{10^6, 10^4, 10^2, 1, 10^{-1}, 10^{-2}, 0\}$. In each iteration, we initialize the model with the solution provided by the previous problem. It is also worth mentioning that method BL-M is warm-started with the solution delivered by BL-R. 

Results in Table \ref{tab:results_BLM_BLR} are intended to compare the relative incomes of BL-M and BL-R for each generating unit whose data is collated in Table \ref{tab:units}. Although method BL-R logically yields lower incomes, the differences with respect to BL-M are below 0.8\%. This means that if model \eqref{eq:prod_bilevel} does not satisfy the conditions to be reformulated as a MIQP or the computational resources are limited, then a good solution (i.e., a solution with a small loss of optimality) can be efficiently computed by solving the regularized NLP version of our approach.  

\begin{table}[]
    \centering
\begin{tabular}{lcc}
    \hline
       & RI$^{\text {BL-M}}$ & RI$^{\text {BL-R}}$ \\
       & (\%) & (\%) \\
    \hline
    Base    & 96.3 & 96.3 \\
    Medium  & 80.0 & 79.2 \\
    Peak    & 58.7 & 58.4 \\
    \hline
\end{tabular}
    \caption{Comparison of BL-M and BL-R (Out-of-sample} results).
    \label{tab:results_BLM_BLR}
\end{table}

Finally, we compare the computational burden of methods FO, DR, BL-M, and BL-R. The average simulation time invested in solving problems \eqref{eq:prod_forecasting}, \eqref{eq:prod_single_gamma} and \eqref{eq:prod_cynthia} for the three generation technologies are indicated in Table \ref{table:times}, where the maximum solution time has been limited to 20 minutes for all methods. These results highlight the higher computational burden required by BL-M to ensure global optimality. On the other hand, the computing times of BL-R are very affordable, especially considering the competitive edge that this method gives to the strategic firm.

\begin{table}
\centering
\begin{tabular}{lcccc}
\hline
 & FO  &  DR  &  BL-R  &  BL-M \\
 & (s)  &  (s)  &  (s)  &  (s) \\
\hline
Base   & 0.24 & 0.65 & 3.90 & 197.77 \\
Medium & 0.35 & 1.06 & 6.80 & 149.89 \\
Peak   & 0.26 & 0.78 & 4.62 &  22.68 \\
\hline
\end{tabular}
\caption{Average computing time.}
\label{table:times}
\end{table}

\section{Conclusions}
\label{sec:conclusions}
In this paper, we have addressed the problem of data-driven decision-making under uncertainty in the presence of contextual information. More precisely, our ultimate purpose has been to construct a parametric model to predict, based on some covariate information, the uncertain parameters that are input to the optimization model by which the decision is made. To this end, we have proposed a bilevel framework whereby such a parametric model is estimated taking into account the impact of its outputs on the feasibility and value of the decision. Under convexity assumptions, we have provided two single-level reformulations of the bilevel program, namely, a non-linear regularized optimization problem and a mixed-integer non-linear reformulation based on the use of large enough constants. When compared to alternative approaches available in the technical literature, ours features two major advantages: it guarantees feasibility in constrained decision-making problems, and its solution can be directly tackled using off-the-shelf optimization solvers under convexity assumptions.


We have theoretically compared our approach with existing ones for three different applications, namely, the newsvendor problem, the product placement problem, and the strategic producer problem. Additionally, we have  evaluated the performance of our approach and its practical relevance through a realistic case study of a strategic producer that participates in the Iberian electricity market. Specifically, numerical results show that our framework not only significantly increases the revenue streams of the firm in general, but also proves to be critical to generation portfolios mainly consisting of peak power units. Indeed, the market revenues of a strategic peak generation portfolio are specially sensitive to the uncertainty in the inverse demand function. Therefore, in this case, the strategic firm may put at risk the bulk of its market incomes, by being left out of the market or trading in deficit. Our approach, however, is, by construction, aware of that sensitivity and thus, is able to retain most of the profit the firm would make under a perfectly predictable inverse demand function.

Potential extensions of this work would include the use of more advanced techniques in the resolution of our bilevel framework such as those employed in more general MPCC problems. Likewise, the generalization of our approach to multi-stage decision-making problems under uncertainty requires further analysis.

\section*{Acknowledgments}

This work was supported in part by the European Research Council (ERC) under the EU Horizon 2020 research and innovation program (grant agreement No. 755705), {\color{black} in part by the Spanish Ministry of Science and Innovation (AEI/10.13039/501100011033) through project PID2020-115460GB-I00, and in part by the Junta de Andalucía (JA), the Universidad de Málaga and the European Regional Development Fund (FEDER) through the research projects P20\_00153 and UMA2018‐FEDERJA‐001.} {\color{black} M. Á. Mu\~{n}oz is also funded by the Spanish Ministry of Science, Innovation and Universities through the State Training Subprogram 2018 of the State Program for the Promotion of Talent and its Employability in R\&D\&I, within the framework of the State Plan for Scientific and Technical Research and Innovation 2017-2020 and by the European Social Fund}. Finally, the authors thankfully acknowledge the computer resources, technical expertise, and assistance provided by the SCBI (Supercomputing and Bioinformatics) center of the University of M\'alaga.

\appendix

\section{Asymptotic consistency}\label{sec:appendix}

{\color{black}

 \begin{prop} \label{prop:consistency} Let $S = \left\lbrace (\alpha'_{i}, \beta'_{i}, x_{i}), \forall i \in \mathcal{N}\right \rbrace$ be an i.i.d sample of size $N$ and suppose that there exists a linear relationship between $\alpha'$ and $\beta' >0$ given by $\frac{\alpha'}{\beta'} = a^{T}x + \xi$, with $\xi$ being a zero-mean noise independent of $x$, $\alpha'$ and $\beta'$, and that the expectations $\mathbb{E}[\alpha']$, $\mathbb{E}[\beta']$ and $\mathbb{E}[\alpha'x]$ are all finite. Then, it almost surely holds in the limit $N \rightarrow \infty$ that  the optimizer of the problem
 \begin{subequations}
\begin{align}
    \min_{w_{\gamma}\in\mathcal{W};\, \hat{q}_{i}} & \enskip \frac{1}{N} \sum_{i \in \mathcal{N}} \beta'_{i} \hat{q}^{2}_{i} - \alpha'_{i} \hat{q}_{i} \label{eq:prod_bilevel_gamma_proof_outer}\\
    \text{s.t.} & \enskip \hat{q}_{i} \in \arg \min_{\underline{q} \le q_{i} \le \overline{q}} q_{i}^{2} - w_{\gamma}^T x_{i} q_{i}, \enskip \forall i \in \mathcal{N}\label{eq:prod_bilevel_gamma_proof_inner}
\end{align} \label{eq:prod_bilevel_gamma_proof}
\end{subequations}
with $\mathcal{W} \subset \mathbb{R}^p$ being a compact set containing $a$, is attained at $w_{\gamma} = a$.
\end{prop}

\begin{proof}

First, notice that $\frac{\alpha'}{\beta'} = a^{T}x + \xi$ implies that $\frac{\mathbb{E}[\alpha'|x]}{\mathbb{E}[\beta'|x]} = a^Tx$, since $\alpha' = \beta'\, a^Tx+\beta' \xi$, and thus, $\mathbb{E}[\alpha'|x] =  a^Tx\, \mathbb{E}[\beta'|x]$ given the independent nature of the noise $\xi$.

The \emph{true} expectation problem associated with the sample average approximation \eqref{eq:prod_bilevel_gamma_proof} is given by:
\begin{subequations}
\begin{align}
\min_{w_{\gamma}\in\mathcal{W};\, \hat{q}(x)} & \enskip \int_{\mathcal{X} \times \mathbb{R}^+ \times \mathbb{R} } \left(\beta' \hat{q}^{2}(x) - \alpha' \hat{q}(x)\right) Q(dx, d\beta',d\alpha')\label{eq:prod_bilevel_gamma_proof_true_ul} \\
    \text{s.t.} & \enskip \hat{q}(x)\in \arg \min_{\underline{q} \leq q \leq \overline{q}} q^{2} - w_{\gamma}^T x q, \enskip \forall x \in \mathcal{X}\label{eq:prod_bilevel_gamma_proof_true_ll}
\end{align} \label{eq:prod_bilevel_gamma_proof_true}
\end{subequations}
where $Q$ is the joint probability law governing the random parameters $\beta'$ and $\alpha'$ and the feature vector $X$.

We first show that $a$ is the unique solution to problem~\eqref{eq:prod_bilevel_gamma_proof_true}. To this end, we note that the lower-level problem \eqref{eq:prod_bilevel_gamma_proof_true_ll} renders the following decision mapping for almost all $x \in \mathcal{X}$:
$$\hat{q}(x)= \max\left(\underline{q}, \min\left(\frac{ w_{\gamma}^T x}{2},\overline{q}\right)\right)$$
which is a continuous function in $w_{\gamma}$.

Now let $Q_X$ be the probability measure of $X$. Consider the following optimization problem, which is a relaxation of \eqref{eq:prod_bilevel_gamma_proof_true}.
\begin{align*}
&\min_{q(x) \in [\underline{q},\overline{q}], \, \forall x \in \mathcal{X}} \enskip \int_{\mathcal{X} \times \mathbb{R}^+ \times \mathbb{R} } \left(\beta' q^{2}(x) - \alpha' q(x)\right) Q(dx, d\beta',d\alpha') =\\
&\min_{q(x) \in [\underline{q},\overline{q}], \, \forall x \in \mathcal{X}} \enskip \int_{\mathcal{X}}   \left(q^{2}(x) \mathbb{E}[\beta'|x]- q(x)\mathbb{E}[\alpha'|x]\right)  Q_X(dx) =\\
&\int_{\mathcal{X}}   \left(\min_{q(x) \in [\underline{q},\overline{q}]} \enskip q^{2}(x) \mathbb{E}[\beta'|x]- q(x)\mathbb{E}[\alpha'|x]\right) Q_X(dx)
\end{align*}
The inner pointwise minimum results in the following optimal decision rule:
$$q(x)= \max\left(\underline{q}, \min\left(\frac{\mathbb{E}[\alpha'|x]}{2\mathbb{E}[\beta'|x]},\overline{q}\right)\right) 
= \max\left(\underline{q}, \min\left(\frac{ a^T x}{2},\overline{q}\right)\right)
$$
for almost all $x \in \mathcal{X}$.

Therefore, since $ w_{\gamma} = a$ is feasible in the true expectation problem \eqref{eq:prod_bilevel_gamma_proof_true}, then it is also an optimal solution to this problem. Furthermore, this solution is unique, if there exists a subset of $\mathcal{X}$ with measure greater than zero such that $\underline{q} < \frac{\mathbb{E}[\alpha'|x]}{2\mathbb{E}[\beta'|x]} < \overline{q}$.

In addition, note that all the samples in $S$ are i.i.d. and that $\beta' q^{2}(x) - \alpha' q(x)$ is dominated by the function $\max \left(\beta' \overline{q}^{2} - \alpha' \overline{q}, \beta' \underline{q}^{2} - \alpha' \underline{q}, \frac{\alpha'^2}{4\beta'}\right)$, which is integrable because the expectations $\mathbb{E}[\alpha']$, $\mathbb{E}[\beta']$ and $\mathbb{E}[\alpha'x]$ are all finite. Indeed, since $\frac{\alpha'}{\beta'} = a^{T}x + \xi$ by assumption, we have that $\mathbb{E}[\frac{\alpha'^2}{4\beta'}] = \frac{a^T}{4} \mathbb{E}[\alpha'x].$

Therefore, by invoking Theorems 5.3 and 7.48  in \cite{shapiro2021lectures}, we have that the minimizer of the sample average approximation problem~\eqref{eq:prod_bilevel_gamma_proof} converges to $a$ almost surely as the sample size $N$ grows to infinity.


%
\end{proof}
}

\section*{CRediT authorship contribution statement}

\textbf{M. A. Muñoz}: Data curation, Software, Methodology, Investigation, Formal analysis, Validation, Writing - Original Draft. \textbf{S. Pineda}: Conceptualization, Methodology, Investigation, Formal analysis, Writing - original draft, Supervision. \textbf{J. M. Morales}: Conceptualization, Methodology, Investigation, Formal analysis, Writing - original draft, Supervision, Funding acquisition.

\bibliographystyle{elsarticle-num-names}
\bibliography{bibliography}

\begin{thebibliography}{40}
\expandafter\ifx\csname natexlab\endcsname\relax\def\natexlab#1{#1}\fi
\providecommand{\url}[1]{\texttt{#1}}
\providecommand{\href}[2]{#2}
\providecommand{\path}[1]{#1}
\providecommand{\DOIprefix}{doi:}
\providecommand{\ArXivprefix}{arXiv:}
\providecommand{\URLprefix}{URL: }
\providecommand{\Pubmedprefix}{pmid:}
\providecommand{\doi}[1]{\href{http://dx.doi.org/#1}{\path{#1}}}
\providecommand{\Pubmed}[1]{\href{pmid:#1}{\path{#1}}}
\providecommand{\bibinfo}[2]{#2}
\ifx\xfnm\relax \def\xfnm[#1]{\unskip,\space#1}\fi
\bibitem[{Bertsimas and Kallus(2020)}]{Bertsimas2020}
\bibinfo{author}{D.~Bertsimas}, \bibinfo{author}{N.~Kallus},
\newblock \bibinfo{title}{From predictive to prescriptive analytics},
\newblock \bibinfo{journal}{Management Science} \bibinfo{volume}{66}
  (\bibinfo{year}{2020}) \bibinfo{pages}{1025--1044}.
\bibitem[{Donti et~al.(2017)Donti, Amos, and Kolter}]{Donti2017}
\bibinfo{author}{P.~Donti}, \bibinfo{author}{B.~Amos}, \bibinfo{author}{J.~Z.
  Kolter},
\newblock \bibinfo{title}{Task-based end-to-end model learning in stochastic
  optimization},
\newblock \bibinfo{journal}{Advances in Neural Information Processing Systems}
  (\bibinfo{year}{2017}) \bibinfo{pages}{5484--5494}.
\bibitem[{Ban and Rudin(2019)}]{Ban2019}
\bibinfo{author}{G.-Y. Ban}, \bibinfo{author}{C.~Rudin},
\newblock \bibinfo{title}{The big data newsvendor: Practical insights from
  machine learning},
\newblock \bibinfo{journal}{Operations Research} \bibinfo{volume}{67}
  (\bibinfo{year}{2019}) \bibinfo{pages}{90--108}.
\bibitem[{Bertsimas et~al.(2019)Bertsimas, McCord, and
  Sturt}]{Bertsimas2019dynamic}
\bibinfo{author}{D.~Bertsimas}, \bibinfo{author}{C.~McCord},
  \bibinfo{author}{B.~Sturt},
\newblock \bibinfo{title}{Dynamic optimization with side information},
\newblock \bibinfo{journal}{arXiv preprint arXiv:1907.07307}
  (\bibinfo{year}{2019}).
\bibitem[{Elmachtoub and Grigas(2021)}]{Elmachtoub2017}
\bibinfo{author}{A.~N. Elmachtoub}, \bibinfo{author}{P.~Grigas},
\newblock \bibinfo{title}{Smart ``{P}redict, then {O}ptimize"},
\newblock \bibinfo{journal}{Management Science}  (\bibinfo{year}{2021}).
  \DOIprefix\doi{https://doi.org/10.1287/mnsc.2020.3922}.
\bibitem[{Esfahani and Kuhn(2018)}]{esfahani2018data}
\bibinfo{author}{P.~M. Esfahani}, \bibinfo{author}{D.~Kuhn},
\newblock \bibinfo{title}{Data-driven distributionally robust optimization
  using the wasserstein metric: Performance guarantees and tractable
  reformulations},
\newblock \bibinfo{journal}{Mathematical Programming} \bibinfo{volume}{171}
  (\bibinfo{year}{2018}) \bibinfo{pages}{115--166}.
\bibitem[{Li et~al.(2019)Li, Jiang, and Mathieu}]{li2019ambiguous}
\bibinfo{author}{B.~Li}, \bibinfo{author}{R.~Jiang}, \bibinfo{author}{J.~L.
  Mathieu},
\newblock \bibinfo{title}{Ambiguous risk constraints with moment and
  unimodality information},
\newblock \bibinfo{journal}{Mathematical Programming} \bibinfo{volume}{173}
  (\bibinfo{year}{2019}) \bibinfo{pages}{151--192}.
\bibitem[{Hao et~al.(2020)Hao, He, Hu, and Jiang}]{Hao2020}
\bibinfo{author}{Z.~Hao}, \bibinfo{author}{L.~He}, \bibinfo{author}{Z.~Hu},
  \bibinfo{author}{J.~Jiang},
\newblock \bibinfo{title}{Robust vehicle pre-allocation with uncertain
  covariates},
\newblock \bibinfo{journal}{Production and Operations Management}
  \bibinfo{volume}{29} (\bibinfo{year}{2020}) \bibinfo{pages}{955--972}.
\bibitem[{Keith and Ahner(2019)}]{keith2019survey}
\bibinfo{author}{A.~J. Keith}, \bibinfo{author}{D.~K. Ahner},
\newblock \bibinfo{title}{A survey of decision making and optimization under
  uncertainty},
\newblock \bibinfo{journal}{Annals of Operations Research}
  (\bibinfo{year}{2019}) \bibinfo{pages}{1--35}.
\bibitem[{Bakker et~al.(2020)Bakker, Dunke, and Nickel}]{Bakker2019structuring}
\bibinfo{author}{H.~Bakker}, \bibinfo{author}{F.~Dunke},
  \bibinfo{author}{S.~Nickel},
\newblock \bibinfo{title}{A structuring review on multi-stage optimization
  under uncertainty: Aligning concepts from theory and practice},
\newblock \bibinfo{journal}{Omega} \bibinfo{volume}{96} (\bibinfo{year}{2020})
  \bibinfo{pages}{102080}.
\bibitem[{Mundru(2019)}]{Mundru2019}
\bibinfo{author}{N.~Mundru}, \bibinfo{title}{Predictive and prescriptive
  methods in operations research and machine learning: an optimization
  approach}, Ph.D. thesis, Massachusetts Institute of Technology,
  \bibinfo{year}{2019}.
\bibitem[{Bertsimas et~al.(2018)Bertsimas, Gupta, and Kallus}]{bertsimas2018}
\bibinfo{author}{D.~Bertsimas}, \bibinfo{author}{V.~Gupta},
  \bibinfo{author}{N.~Kallus},
\newblock \bibinfo{title}{Robust sample average approximation},
\newblock \bibinfo{journal}{Mathematical Programming} \bibinfo{volume}{171}
  (\bibinfo{year}{2018}) \bibinfo{pages}{217--282}.
\bibitem[{Vives(1984)}]{Vives1984}
\bibinfo{author}{X.~Vives},
\newblock \bibinfo{title}{Duopoly information equilibrium: Cournot and
  {B}ertrand},
\newblock \bibinfo{journal}{Journal of Economic Theory} \bibinfo{volume}{34}
  (\bibinfo{year}{1984}) \bibinfo{pages}{71 -- 94}.
\bibitem[{Wu et~al.(2008)Wu, Zhai, and Huang}]{wu2008incentives}
\bibinfo{author}{J.~Wu}, \bibinfo{author}{X.~Zhai}, \bibinfo{author}{Z.~Huang},
\newblock \bibinfo{title}{Incentives for information sharing in duopoly with
  capacity constraints},
\newblock \bibinfo{journal}{Omega} \bibinfo{volume}{36} (\bibinfo{year}{2008})
  \bibinfo{pages}{963--975}.
\bibitem[{Bimpikis et~al.(2019)Bimpikis, Ehsani, and
  Ilk{\i}l{\i}{\c{c}}}]{Bimpikis2019}
\bibinfo{author}{K.~Bimpikis}, \bibinfo{author}{S.~Ehsani},
  \bibinfo{author}{R.~Ilk{\i}l{\i}{\c{c}}},
\newblock \bibinfo{title}{Cournot competition in networked markets},
\newblock \bibinfo{journal}{Management Science} \bibinfo{volume}{65}
  (\bibinfo{year}{2019}) \bibinfo{pages}{2467--2481}.
\bibitem[{Allaz and Vila(1993)}]{Allaz1993}
\bibinfo{author}{B.~Allaz}, \bibinfo{author}{J.-L. Vila},
\newblock \bibinfo{title}{Cournot competition, forward markets and efficiency},
\newblock \bibinfo{journal}{Journal of Economic Theory} \bibinfo{volume}{59}
  (\bibinfo{year}{1993}) \bibinfo{pages}{1 -- 16}.
\bibitem[{Ruiz et~al.(2008)Ruiz, Conejo, and García-Bertrand}]{Ruiz2008}
\bibinfo{author}{C.~Ruiz}, \bibinfo{author}{A.~Conejo},
  \bibinfo{author}{R.~García-Bertrand},
\newblock \bibinfo{title}{Some analytical results pertaining to {C}ournot
  models for short-term electricity markets},
\newblock \bibinfo{journal}{Electric Power Systems Research}
  \bibinfo{volume}{78} (\bibinfo{year}{2008}) \bibinfo{pages}{1672 -- 1678}.
\bibitem[{Mandi et~al.(2019)Mandi, Demirovi{\'c}, Stuckey, Guns
  et~al.}]{Mandi2019}
\bibinfo{author}{J.~Mandi}, \bibinfo{author}{E.~Demirovi{\'c}},
  \bibinfo{author}{P.~Stuckey}, \bibinfo{author}{T.~Guns}, et~al.,
\newblock \bibinfo{title}{Smart predict-and-optimize for hard combinatorial
  optimization problems},
\newblock \bibinfo{journal}{arXiv preprint arXiv:1911.10092}
  (\bibinfo{year}{2019}).
\bibitem[{hao Kao et~al.(2009)hao Kao, Roy, and Yan}]{Kao2009}
\bibinfo{author}{Y.~hao Kao}, \bibinfo{author}{B.~V. Roy},
  \bibinfo{author}{X.~Yan},
\newblock \bibinfo{title}{Directed regression},
\newblock \bibinfo{journal}{Advances in Neural Information Processing Systems
  22}  (\bibinfo{year}{2009}) \bibinfo{pages}{889--897}.
\bibitem[{Wilder et~al.(2019)Wilder, Dilkina, and Tambe}]{Wilder2019}
\bibinfo{author}{B.~Wilder}, \bibinfo{author}{B.~Dilkina},
  \bibinfo{author}{M.~Tambe},
\newblock \bibinfo{title}{Melding the data-decisions pipeline: Decision-focused
  learning for combinatorial optimization},
\newblock \bibinfo{journal}{Proceedings of the AAAI Conference on Artificial
  Intelligence} \bibinfo{volume}{33} (\bibinfo{year}{2019})
  \bibinfo{pages}{1658--1665}.
\bibitem[{Dempe(2017)}]{Dempe2017}
\bibinfo{author}{S.~Dempe}, \bibinfo{title}{Bilevel Optimization: Reformulation
  and First Optimality Conditions}, \bibinfo{publisher}{Springer Singapore},
  \bibinfo{address}{Singapore}, \bibinfo{year}{2017}, pp.
  \bibinfo{pages}{1--20}.
\bibitem[{Casorr{\'a}n et~al.(2019)Casorr{\'a}n, Fortz, Labb{\'e}, and
  Ord{\'o}{\~n}ez}]{Labbe2019}
\bibinfo{author}{C.~Casorr{\'a}n}, \bibinfo{author}{B.~Fortz},
  \bibinfo{author}{M.~Labb{\'e}}, \bibinfo{author}{F.~Ord{\'o}{\~n}ez},
\newblock \bibinfo{title}{A study of general and security {S}tackelberg game
  formulations},
\newblock \bibinfo{journal}{European Journal of Operational Research}
  \bibinfo{volume}{278} (\bibinfo{year}{2019}) \bibinfo{pages}{855--868}.
\bibitem[{Kleywegt et~al.(2002)Kleywegt, Shapiro, and Homem-de
  Mello}]{kleywegt2002sample}
\bibinfo{author}{A.~J. Kleywegt}, \bibinfo{author}{A.~Shapiro},
  \bibinfo{author}{T.~Homem-de Mello},
\newblock \bibinfo{title}{The sample average approximation method for
  stochastic discrete optimization},
\newblock \bibinfo{journal}{SIAM Journal on Optimization} \bibinfo{volume}{12}
  (\bibinfo{year}{2002}) \bibinfo{pages}{479--502}.
\bibitem[{Ruszczynski and Shapiro(2003)}]{ruszczynski2003stochastic}
\bibinfo{author}{A.~Ruszczynski}, \bibinfo{author}{A.~Shapiro},
  \bibinfo{title}{Stochastic programming (handbooks in operations research and
  management science)}, \bibinfo{year}{2003}.
\bibitem[{Ho and Hanasusanto(2019)}]{Ho2019}
\bibinfo{author}{C.~P. Ho}, \bibinfo{author}{G.~A. Hanasusanto},
\newblock \bibinfo{title}{On data-driven prescriptive analytics with side
  information: A regularized {N}adaraya-{W}atson approach},
\newblock \bibinfo{journal}{URL:
  http://www.optimization-online.org/DB\_FILE/2019/01/7043.pdf}
  (\bibinfo{year}{2019}).
\bibitem[{Bertsimas and Van~Parys(2017)}]{Bertsimas2017}
\bibinfo{author}{D.~Bertsimas}, \bibinfo{author}{B.~Van~Parys},
\newblock \bibinfo{title}{Bootstrap robust prescriptive analytics},
\newblock \bibinfo{journal}{arXiv preprint arXiv:1711.09974}
  (\bibinfo{year}{2017}).
\bibitem[{Diao and Sen(2020)}]{Diao2020}
\bibinfo{author}{S.~Diao}, \bibinfo{author}{S.~Sen},
\newblock \bibinfo{title}{Distribution-free algorithms for learning enabled
  optimization with non-parametric estimation},
\newblock \bibinfo{journal}{Management Science} \bibinfo{volume}{66}
  (\bibinfo{year}{2020}) \bibinfo{pages}{1025--1044}.
\bibitem[{Boyd et~al.(2004)Boyd, Boyd, and Vandenberghe}]{Boyd2004}
\bibinfo{author}{S.~Boyd}, \bibinfo{author}{S.~P. Boyd},
  \bibinfo{author}{L.~Vandenberghe}, \bibinfo{title}{Convex optimization},
  \bibinfo{publisher}{Cambridge university press}, \bibinfo{year}{2004}.
\bibitem[{Scheel and Scholtes(2000)}]{Scheel2000}
\bibinfo{author}{H.~Scheel}, \bibinfo{author}{S.~Scholtes},
\newblock \bibinfo{title}{Mathematical programs with complementarity
  constraints: Stationarity, optimality, and sensitivity},
\newblock \bibinfo{journal}{Mathematics of Operations Research}
  \bibinfo{volume}{25} (\bibinfo{year}{2000}) \bibinfo{pages}{1--22}.
\bibitem[{Scholtes(2001)}]{Scholtes2001}
\bibinfo{author}{S.~Scholtes},
\newblock \bibinfo{title}{Convergence properties of a regularization scheme for
  mathematical programs with complementarity constraints},
\newblock \bibinfo{journal}{SIAM Journal on Optimization} \bibinfo{volume}{11}
  (\bibinfo{year}{2001}) \bibinfo{pages}{918--936}.
\bibitem[{Ralph and Wright(2004)}]{Ralph2004}
\bibinfo{author}{D.~Ralph}, \bibinfo{author}{S.~J. Wright},
\newblock \bibinfo{title}{Some properties of regularization and penalization
  schemes for mpecs},
\newblock \bibinfo{journal}{Optimization Methods and Software}
  \bibinfo{volume}{19} (\bibinfo{year}{2004}) \bibinfo{pages}{527--556}.
\bibitem[{Fortuny-Amat and McCarl(1981)}]{Fortuny1981}
\bibinfo{author}{J.~Fortuny-Amat}, \bibinfo{author}{B.~McCarl},
\newblock \bibinfo{title}{A representation and economic interpretation of a
  two-level programming problem},
\newblock \bibinfo{journal}{Journal of the operational Research Society}
  \bibinfo{volume}{32} (\bibinfo{year}{1981}) \bibinfo{pages}{783--792}.
\bibitem[{Pineda et~al.(2018)Pineda, Bylling, and Morales}]{Pineda2018}
\bibinfo{author}{S.~Pineda}, \bibinfo{author}{H.~Bylling},
  \bibinfo{author}{J.~Morales},
\newblock \bibinfo{title}{Efficiently solving linear bilevel programming
  problems using off-the-shelf optimization software},
\newblock \bibinfo{journal}{Optimization and Engineering} \bibinfo{volume}{19}
  (\bibinfo{year}{2018}) \bibinfo{pages}{187--211}.
\bibitem[{Garcia et~al.(2021)Garcia, Street, de~Mello, and Muñoz}]{Garcia2021}
\bibinfo{author}{J.~D. Garcia}, \bibinfo{author}{A.~Street},
  \bibinfo{author}{T.~H. de~Mello}, \bibinfo{author}{F.~D. Muñoz},
  \bibinfo{title}{Application-driven learning via joint prediction and
  optimization of demand and reserves requirement}, \bibinfo{year}{2021}.
  \href{http://arxiv.org/abs/2102.13273}{{\tt arXiv:2102.13273}}.
\bibitem[{{D}jurovi{\'c} et~al.(2012){D}jurovi{\'c}, Mila{\v{c}}i{\'c}, and
  Kr{\v{s}}ulja}]{Djurovic2012}
\bibinfo{author}{M.~{\v{Z}}. {D}jurovi{\'c}},
  \bibinfo{author}{A.~Mila{\v{c}}i{\'c}}, \bibinfo{author}{M.~Kr{\v{s}}ulja},
\newblock \bibinfo{title}{A simplified model of quadratic cost function for
  thermal generators},
\newblock \bibinfo{journal}{Proceedings of the 23rd International DAAAM
  Symposium} \bibinfo{volume}{23} (\bibinfo{year}{2012})
  \bibinfo{pages}{25--28}.
\bibitem[{Omi(2020)}]{Omie}
\bibinfo{title}{{OMIE offer and demand aggregated curves}},
  \bibinfo{year}{2020}. \bibinfo{note}{Available (online):
  \url{https://www.omie.es/}. Last accessed on February 6, 2020}.
\bibitem[{Ent(2020)}]{Entsoe}
\bibinfo{title}{{ENTSO-E. Transparency Platform}}, \bibinfo{year}{2020}.
  \bibinfo{note}{Available (online): \url{https://transparency.entsoe.eu/}.
  Last accessed on February 6, 2020}.
\bibitem[{CPL(2020)}]{CPLEX}
\bibinfo{title}{{CPLEX Optimizer}}, \bibinfo{howpublished}{IBM},
  \bibinfo{year}{2020}. \bibinfo{note}{Available (online):
  \url{https://www.ibm.com/analytics/cplex-optimizer}}.
\bibitem[{CON(2020)}]{CONOPT}
\bibinfo{title}{{CONOPT Optimizer}}, \bibinfo{howpublished}{ARKI Consulting
  Development A/S}, \bibinfo{year}{2020}. \bibinfo{note}{Available (online):
  \url{http://www.conopt.com/}}.
\bibitem[{Shapiro et~al.(2021)Shapiro, Dentcheva, and
  Ruszczynski}]{shapiro2021lectures}
\bibinfo{author}{A.~Shapiro}, \bibinfo{author}{D.~Dentcheva},
  \bibinfo{author}{A.~Ruszczynski}, \bibinfo{title}{Lectures on stochastic
  programming: modeling and theory}, \bibinfo{publisher}{SIAM},
  \bibinfo{year}{2021}.

\end{thebibliography}

\end{document}